\documentclass[10pt]{article}
\usepackage{amsmath}
\usepackage{amsfonts}
\usepackage{mathrsfs}
\usepackage{bbm}
\usepackage{amsfonts}
\usepackage{enumitem}
\usepackage{amssymb,amsmath,amsthm,graphicx}
\usepackage{float} \usepackage[colorlinks=true]{hyperref}
\hypersetup{urlcolor=blue, citecolor=red}
\def\Xint#1{\mathchoice
{\XXint\displaystyle\textstyle{#1}}%
{\XXint\textstyle\scriptstyle{#1}}%
{\XXint\scriptstyle\scriptscriptstyle{#1}}%
{\XXint\scriptscriptstyle\scriptscriptstyle{#1}}%
\!\int}
\def\XXint#1#2#3{{\setbox0=\hbox{$#1{#2#3}{\int}$ }
\vcenter{\hbox{$#2#3$ }}\kern-.6\wd0}}

\def\dashint{\Xint-}

\allowdisplaybreaks

\begin{document}
\title{Local behaviour of the mixed local and nonlocal problems with nonstandard growth
\thanks{Supported by National Natural
 Science Foundation of China (No. 12071098), Postdoctoral Science Foundation of Heilongjiang Province (No. LBH-Z22177), National Postdoctoral Program for Innovative Talents of China (No. BX20220381)}}
\author{Mengyao Ding$^{1}$,~~Yuzhou Fang$^{2}$,~~Chao Zhang$^{2, 3}$\thanks{Corresponding author.
E-Mail:
myding@pku.edu.cn (M. Ding),
18b912036@hit.edu.cn (Y. Fang),
 czhangmath@hit.edu.cn (C. Zhang)}\\
{\small $^{1}$ School of Mathematical Sciences, Peking University,
Beijing 100871, PR China} \\
{\small $^{2}$ School of Mathematics, Harbin Institute of Technology, Harbin 150001, PR China}\\
{\small $^{3}$ Institute for Advanced Study in Mathematics, Harbin Institute of Technology, Harbin 150001, PR China}\\
}
\date{}
\newtheorem{theorem}{Theorem}
\newtheorem{definition}{Definition}[section]
\newtheorem{lemma}{Lemma}[section]
\newtheorem{proposition}{Proposition}[section]
\newtheorem{corollary}{Corollary}\newtheorem{remark}{Remark}
\renewcommand{\theequation}{\thesection.\arabic{equation}}
\catcode`@=11 \@addtoreset{equation}{section} \catcode`@=12
\maketitle{}

\begin{abstract}
We consider the mixed local and nonlocal functionals with nonstandard growth
\begin{eqnarray*}
u\mapsto\int_{\Omega}(|Du|^p-f(x)u)\,dx+\int_{\mathbb{R}^N}\int_{\mathbb{R}^N}\frac{|u(x)-u(y)|^q}{|x-y|^{N+sq}}\,dxdy
\end{eqnarray*}
with $1<p\le sq$, $0<s<1$ and $\Omega\subset\mathbb{R}^N$ being a bounded domain. We study, by means of expansion of positivity, local behaviour of the minimizers of such problems, involving local boundedness, local H\"{o}lder continuity and Harnack inequality. The results above can be seen as a natural extension of the results under the center condition that $p>sq$ in [De Filippis-Mingione, Math. Ann., https://doi.org/10.1007/s00208-022-02512-7].
\begin{description}
\item[2010MSC:] 35B45; 35B65; 35R11; 35K55
\item[Keywords:]Local regularity; Mixed local and nonlocal functionals; Nonstandard growth; Caccioppoli estimates; Expansion of positivity
\end{description}
\end{abstract}



\section{Introduction}
In the present paper, we aim at investigating local regularity for the minimizers of the following functional:
\begin{eqnarray}\label{q1}
\mathcal{E}(u ; \Omega):=\int_{\Omega}(F(x,Du)-f(x)u)\,dx+\iint_{\mathcal{C}_{\Omega}} |u(x)-u(y)|^qK_{s q}(x, y)\,dxdy,
\end{eqnarray}
where $\Omega\subset\mathbb{R}^N$ ($N\ge2$) is a bounded domain and
$$
\mathcal{C}_{\Omega}:=(\mathbb{R}^N \times \mathbb{R}^N)\setminus((\mathbb{R}^N \times \Omega)\times(\mathbb{R}^N \times \Omega)).
$$
The Carath\'{e}odory function $F: \Omega \times \mathbb{R}^N \rightarrow \mathbb{R}$ satisfies
\begin{align}\label{F}
\Lambda^{-1}|\xi|^p \leq F(x, \xi) \leq \Lambda|\xi|^p
\end{align}
for some $\Lambda>1$. The symmetric kernel $K_{s q}: \mathbb{R}^N \times \mathbb{R}^N\rightarrow \mathbb{R}$ is a measurable function taking the form
\begin{align}\label{K}
K_{s q}(x, y)=\frac{a(x,y)}{|x-y|^{N+s q}}
\end{align}
with $a(x,y)=a(y,x)$ satisfying
\begin{align}\label{ca}
a_0\le a(x,y)\le A_0\quad \text {for all } x,y \in \mathbb{R}^N \times \mathbb{R}^N,
\end{align}
where the constants $a_0\ge0$ and $A_0>0$. 
In addition,
\begin{equation}\label{exp}
1<p\le sq, \quad s\in(0,1).
\end{equation}
In order to state the definition of minimizers, we need to define a function space
$$
\mathcal{A}(\Omega):=\left\{v: \mathbb{R}^N \rightarrow \mathbb{R} ~\Big| ~v|_{\Omega} \in L^p(\Omega)
\text { and } \mathcal{E}(v ; \Omega)<\infty\right\},
$$
and recall the tail space
$$
L_{sq}^{q-1}\left(\mathbb{R}^{N}\right):=\left\{v \in L_{\operatorname{loc}}^{q-1}\left(\mathbb{R}^{N}\right)\Big| \int_{\mathbb{R}^{N}} \frac{|v(x)|^{q-1}}{1+|x|^{N+sq}} d x<\infty\right\}, 
$$
with a corresponding nonlocal tail defined by
\begin{align*}
\operatorname{Tail}(v ; x_{0}, r)
&:=\left(r^{sq}\int_{\mathbb{R}^{N} \backslash B_{r}\left(x_{0}\right)} \frac{|v(x)|^{q-1}}{\left|x-x_{0}\right|^{N+s q}} d x\right)^{\frac{1}{q-1}} \quad \text{for } r>0, x_0\in\mathbb{R}^N.
\end{align*}
From these definitions, it is easy to deduce that $\operatorname{Tail}(v; x_{0}, r)$ is well-defined for any
$v\in L^{q-1}_{sq}(\Bbb{R}^N)$.

Now we are in a position to present the definition of minimizers to \eqref{q1} as follows.

\medskip

\noindent{\bf Definition 1.1.}
 We say that $u \in \mathcal{A}(\Omega)$ is a minimizer of $\mathcal{E}$ if
\begin{align*}
\mathcal{E}(u ; \Omega) \leq \mathcal{E}(v ; \Omega)
\end{align*}
for any measurable function $v: \mathbb{R}^N \rightarrow \mathbb{R}$ with $v=u$ a.e. in $\mathbb{R}^N \backslash \Omega$.

Next, we will introduce some related results provided by the existing literature in the
coming subsection.

\subsection{Overview of related literature}

Mixed local and nonlocal problems, as a recently emerging subject, have already attracted an intensive attention. Before introducing results on such mixed problems,
we first refer the readers to \cite{Ca11,Ca07,Ca09,S2007,S2006} for the regularity of integro-differential equations,
and \cite{BLS2018,BLS2019,Co,DKP2014,DKP2015, K2007,K2009,MSY21,No23} for the properties of weak solutions to fractional $p$-Laplacian.  
Here, we emphasize the approach developed in \cite{Co}, where Cozzi \cite{Co} introduced fractional De Giori classes and then exhibited regularity theory for these classes.
As immediate applications of regularity results, the boundedness, H\"{o}lder continuity and Harnack inequality directly can be established
for the weak solutions and minimizers of corresponding equations/functionals via proving nonlocal Caccioppoli estimates.

Let us turn to the combination of local and nonlocal linear equations of the simplest version:
\begin{align}\label{q2}
-\Delta u+(-\Delta)^s u=0,
\end{align}
for which the Harnack inequality of nonnegative harmonic functions was showed in \cite{Fo09}; see \cite{Chen12} for the boundary version.
In addition, the Harnack inequality regarding the parabolic version of \eqref{q2} was established in \cite{BBCK9,CK10}, where however the authors only proved such inequality for globally nonnegative solutions. Very recently, Garain-Kinnunen \cite{GK21} proved a weak Harnack inequality with a tail term for sign changing solutions to the parabolic problem of \eqref{q2}. For what concerns radial symmetry, maximum principle, interior and boundary Lipschitz regularity for the weak solutions to \eqref{q2}, Biagi-Dipierro-Valdinoci-Vecchi \cite{BDVV,BDVV'} have systematically studied.  One can also refer to \cite{Su22} and \cite{BMS23}  for  the qualitative and quantitative properties on the equations involving a general nonhomogenous term $g(u,x)$
and the boundary regularity of more general mixed local‐nonlocal operators. 

When it comes to the nonlinear framework of \eqref{q2} as below,
  \begin{align}\label{q3}
-\Delta_p u+(-\Delta)_{p}^{s} u=0,
\end{align}
by using a purely analytic approach based on the De Giorgi-Nash-Moser theory, Garain-Kinnunen \cite{Ki22} demonstrated
local boundedness and H\"{o}lder continuity, Harnack inequalities as well as semicontinuity representative of weak solutions. Consequently, Fang-Shang-Zhang \cite{FSZ22} studied the local regularity for the parabolic counterpart to \eqref{q3}; see \cite{GK2023} for more general construction. Some extra aspects of such equations have already been explored as well, such as existence, uniqueness and higher H\"{o}lder regularity \cite{GL23}, problems with measure data \cite{BS23}. 
It is worth mentioning that De Filippis-Mingione \cite{Fi22} first considered the mixed local and nonlocal functionals with nonuniform growth of the type \eqref{q1} and established the maximal regularity, that is (local) $C^{1,\alpha}$-regularity, for minimizers of such problems under the center assumption $p>sq$.
Byun-Lee-Song \cite{By22} investigated the case that local term provides no regularizing effects, where the H\"{o}lder regularity and Harnack inequality were discussed for the minimizers to the functionals modeled after
\begin{eqnarray*}
u\mapsto\int_{\mathbb{R}^N}\int_{\mathbb{R}^N}\frac{|u(x)-u(y)|^q}{|x-y|^{N+sq}}\,dxdy+\int_{\Omega}a(x)|Du|^p\,dx
\end{eqnarray*}
with $a(x)\ge0, p>q>1$. Note that the hypothesis required in \eqref{exp} and the center assumption on growth exponents $p,q$ in \cite{Fi22,By22} are mutually exclusive. Finally, let us mention that in view of the condition $p\neq q$ the functional \eqref{q1} is closely related to a large number of anisotropic local or nonlocal problems with nonstandard growth, see for instance \cite{Mar89,CM15,BM20,BS20,DeFP19,BKO,FZ}.

\subsection{Statements of the main results}

We would like to point out that the present work is motivated by  the paper of De Filippis and Mingione \cite{Fi22}. \emph{In particular, the authors in the open problem 8.3 of  \cite{Fi22}  posed a question concerning the possibility of removing, when $q>p$, the (local) boundedness assumption on $u$ in Theorems 3, 5 and 6.} To this end, we try to give a positive answer to that. As far as we known, there is no theory yet for the mixed local and nonlocal equation \eqref{q1}
when $sq\ge p$. Our intention is to study the local boundedness, H\"{o}lder continuity and Harnack estimat for the minimizers of \eqref{q1} with nonhomogeneous term $f$ through the expansion of positivity.
In order to simplify our presentations, we introduce some notations needed in the coming text.

\medskip

\noindent{\bf Notations.}
As usual, the domain $B_\rho(x_0)$ is a ball with radius $\rho>0$
 and center $x_0\in\Bbb{R}^N$.
 The symbol can be simplified by writing $B_{\rho}=B_{\rho}(x_0)$.
For any $x_0\in\Bbb{R}^N$, $k\in\Bbb{R}$ and $r>0$, we define the sets
\begin{equation}
\label{A}
A^-(k,x_0,r):=B_{r}(x_0)\cap\big\{x\in\Bbb{R}^N|~u<k\big\} \quad \text{and} \quad A^+(k,x_0,r):=B_{r}(x_0)\cap\big\{x\in\Bbb{R}^N|~u> k\big\},
\end{equation}
which are abbreviated by
$$A^-(k,r):=A^-(k,x_0,r) \quad \text{and} \quad A^+(k,r):= A^+(k,x_0,r)$$
when the choice of $x_0$ is clear.
For $g \in L^{1}(V)$, the mean average of $g$ is given by
$$
(g)_{V} :=\dashint_{V} g(x) \,d x :=\frac{1}{|V|} \int_{V} g(x) \,dx.
$$ 
For the function $a(\cdot)$, we denote
$$a_{V}^+:=\sup_{x\in V}a(x)$$
for any $V\subseteq \Omega$. For the case $V=B_{R}(x_0)$, we simply the notations by taking $a_{R}^+:=a_{B_R(x_0)}^+$.

The continuous measure $\mu$ in this work admits
$$d\mu= d\mu(x,y) = |x-y|^{-N-sq}\,dxdy.$$
Throughout this paper, we set $p^*:=\frac{Np}{N-p}  \text { for } 1<p<N$, $q_s^*:=\frac{Nq}{N-sq}  \text { for } 1<sq<N$
and $\bar{q}:=\frac{q(\gamma-1)}{\gamma}.$ For arbitrarily fixed $r>0$, we also need define the functions
\begin{align}\label{ha}
h_r(t):=\frac{ t^{p-1} }{ r^p }+\frac{ t^{q-1} }{r^{sq}},\qquad g_r(t):=\frac{ t^{\bar{q}-1} }{ r^{\bar{q}} },\qquad t\ge 0
\end{align}
and
\begin{align}\label{Ha}
H_r(t):=\frac{ t^{p} }{ r^p }+\frac{ t^{q} }{r^{sq}},\qquad t\ge0.
\end{align}
Throughout this paper, $H_{r}^{-1},h_{r}^{-1},g_{r}^{-1}$ are separately the inverses of $H_{r},h_{r},g_{r}$, and we use $C$ to denote a general positive constant which only depends on $p,s,q,\gamma,\Lambda,N$ and $a(\cdot,\cdot)$ without additional explanations.
\\

Now, we are in a position to present the boundedness results.
\begin{theorem}
[\label{th1}Local boundedness]
Assume that \eqref{F}--\eqref{exp} hold 
with $a_0>0$.
Suppose that $f \in L_{\operatorname{loc}}^\gamma(\Omega)$ with
$\gamma>\max\Big\{\frac{N}{p},\frac{q}{q-1}\Big\}$.
Let $B_r:=B_r(x_0)\subset\subset\Omega$. Then for any minimizer $u \in \mathcal{A}(\Omega) \cap L_{s q}^{q-1}\left(\mathbb{R}^N\right)$ of the functional \eqref{q1}, the following estimate holds
for any $\delta>0$,
\begin{align*}
\sup _{B_{r/2}} u
\leq C_\delta H_{r}^{-1}\left(\dashint_{B_{ r}} H_{r}\left(u_{+}\right) d x\right)
+\delta h_{r}^{-1}\left(r^{-sq}\operatorname{Tail}^{q-1} \left(u_+;x_0,\frac{r}{2}\right)\right)
+\delta g_{r}^{-1}\left(\Big(\dashint_{B_{r}}|f|^\gamma\,dx\Big)^{\frac{1}{\gamma}}\right)
\end{align*}
 with $C_{\delta}>0$ depending only on $s,p,q,N,\Lambda,\gamma,a_0,A_0$ and $\delta$.
\end{theorem}
Observe that in the case $a_0>0$, we can see $K_{s q}(x, y)\approx\frac{1}{|x-y|^{N+s q}}$ as usual. Nonetheless, when the function $a(\cdot,\cdot)$ vanishes on some points (i.e., $a_0=0$), the functional \eqref{q1} exhibits the features of mixed local and nonlocal double phase functional in fact. At this point. we still can prove minimizers are locally bounded by confining the distance between $p$ and $q$ as below. In addition, taking into account the nonlocal double integral in \eqref{q1} whose kernel $K_{sq}(\cdot,\cdot)$ is perturbed by coefficient $a$, here we introduce a ``tail space with weight" and the corresponding nonlocal tail with weight denoted by
$$
L^{q-1}_{a,sq}(\Omega,\mathbb{R}^N)=\left\{v\in L^{q-1}_{\rm loc}(\mathbb{R}^N): \underset{x\in \Omega}{{\rm ess}\sup}\int_{\mathbb{R}^N}a(x,y)\frac{|v(y)|^{q-1}}{(1+|y|)^{N+sq}}\,dy<\infty\right\}
$$
and
\begin{align}\label{taila}
\operatorname{Tail}_a(v ; x_{0}, r)
&:=\left(r^{sq}\underset{x\in B_{2r}(x_0)}{{\rm ess}\sup}\int_{\mathbb{R}^{N} \backslash B_{r}\left(x_{0}\right)}a(x,y) \frac{|v(y)|^{q-1}}{\left|y-x_{0}\right|^{N+s q}}\, d y\right)^{\frac{1}{q-1}}.
\end{align}
We can readily verify that $\mathrm{Tail}_a(u;x_0,r)$ is finite for every $x_0\in\Omega$ and $r\in(0,2^{-1}\mathrm{dist\{x_0,\partial\Omega\}})$, provided that $u\in L_{a,sq}^{q-1}\left(\Omega,\mathbb{R}^N\right)$. Thereby when we deduce the local boundedness on minimizers, we only impose local boundedness in $\Omega\times\Omega$ on $a(x,y)$.
\begin{theorem} [\label{th2}Local boundedness]
Suppose that \eqref{F}--\eqref{K} and \eqref{exp} as well as $0\le a(x,y)\in L^\infty_{\rm loc}(\Omega\times\Omega)$ are in force. Let
\begin{align}\label{pqs}
\left\{\begin{array}{l}
\frac{q}{p} \leq 1+\frac{p}{N-p} \text{  for  } 1<p\le N,\\[2mm]
q<\infty\text{  for  } p> N,
\end{array}\right.
\end{align}
and $f \in L_{\operatorname{loc}}^\gamma(\Omega)$ with
\begin{align}\label{gam}
\gamma>\max \left\{\frac{q}{q-p},\frac{p}{p-1}, \frac{N}{p}\right\} \ \text{for } p\le N, \ \text{and } \gamma=1 \ \text{for } p> N.
\end{align}
Then any minimizer $u \in \mathcal{A}(\Omega) \cap L_{a,sq}^{q-1}\left(\Omega,\mathbb{R}^N\right)$ of the functional \eqref{q1} is locally bounded in $\Omega$.
\end{theorem}

Based on the above boundedness result, Theorem \ref{th1}, we can further apply the expansion of positivity technique to establish the H\"{o}lder continuity of minimizers.
We remark that the coming H\"{o}lder continuity result is built on the assumption $a_0>0$. The main reason is  that when $a(x,y)\ge0$, the constant $C$ in the Lemma \ref{lem4.1} and Lemma \ref{lem4.2} would depend on
$L^\infty$-norm of minimizers and also on $M$ in Lemma \ref{lem4.2},
even if we impose the H\"{o}lder continuity on $a(\cdot,\cdot)$.
This would lead to the failure of the iteration in the proof of Theorem \ref{th3}.

\begin{theorem} [\label{th3}H\"{o}lder continuity]
Assume that \eqref{F}--\eqref{exp} hold 
with $a_0>0$.
Let $f \in L_{\operatorname{loc}}^\gamma(\Omega)$ with 
$\gamma>\max\Big\{\frac{N}{p},\frac{q}{q-1}\Big\}$.
Then any minimizer $u \in \mathcal{A}(\Omega) \cap L_{s q}^{q-1}\left(\mathbb{R}^N\right)$ of the functional \eqref{q1} is locally $C^{0,\alpha}$-continuous in $\Omega$ with some $\alpha\in(0,1)$ depending only upon $N,p,s,q,\Lambda,\gamma,a_0,A_0$.
\end{theorem}

Additionally, according to the supremum estimate in Theorem \ref{th1}, tail estimate below along with expansion of positivity, we could deduce the Harnack inequality for minimizers. However, when $a(x,y)\ge0$, it is not clear how to control $\operatorname{Tail}_a(u_{+} ; x_{0}, R)$ by $\operatorname{Tail}_a(u_{-} ; x_{0}, R)$.
Hence, we impose the positivity assumption on $a_0$ to ensure tail term in the Harnack inequality involves $u_{-}$ instead of $u$.
\begin{theorem} [\label{th4}Harnack inequality]
Assume that \eqref{F}--\eqref{exp} hold 
with $a_0>0$.
Suppose that $f \in L_{\operatorname{loc}}^\gamma(\Omega)$ with 
$\gamma>\max\Big\{\frac{N}{p},\frac{q}{q-1}\Big\}$. Let $u \in \mathcal{A}(\Omega) \cap L_{s q}^{q-1}\left(\mathbb{R}^N\right)$ is a minimizer of the functional \eqref{q1} such that
$u\ge 0$ in $B_{2R}(x_0)\subset\Omega$.  
Assume $R\in(0,1)$ satisfies $0<R<\operatorname{dist}\left(x_0, \partial \Omega\right) / 4$. Then there holds that
$$\sup_{B_R(x_0)} u \le  C\Big(\inf_{B_R(x_0)} u +\operatorname{Tail}_a(u_{-} ; x_{0}, R)+
 (R^{  p-N/\gamma  }\|f\|_{L^\gamma(B_{2R})})^{\frac{1}{p-1}}
+(  R^{   \bar{q}-N/\gamma  }\|f\|_{L^\gamma(B_{2R})})^{\frac{1}{\bar{q}-1}}\Big)$$
with $C$ only depending on $\Lambda,s,p,q,N,\gamma,a_0,A_0$.
\end{theorem}

\subsection{Preliminaries}

This section collects some imbedding inequalities and iteration results as preliminary ingredients. The first one is fractional Poincar\'{e} inequality in $W^{s, q}$. From now on, unless otherwise specified, we always suppose the conditions \eqref{F}--\eqref{exp} hold true.

\begin{lemma}\label{lem-Poincare}{\rm (see
\cite[Formula (6.3)]{M03})}
 Let $s \in(0,1)$ and $q\in[1,\infty)$. Then for any $f \in W^{s, q}\left(B_{r}\right)$,
there holds that
$$
\dashint_{B_{r}}\left|f(x)-(f)_{B_{r}}\right|^{q}\,dx
\leq C r^{sq-N}\int_{B_{r}} \int_{B_{r}} \frac{ |f(x)-f(y)|^{q}}{|x-y|^{N+sq }}\,dxdy
$$
with $C>0$ only depending on $s,q$ and $N$.
\end{lemma}

What follows is a version of isoperimetric inequality for level sets of Sobolev functions. This kind of inequality was first introduced by De Giorgi \cite{De}. Here we invoke the one formulated by \cite{Co}.
\begin{lemma}\label{measure1}
Let $N \geq 2$ be an integer and $p>1$. Then, for any two real numbers $h_1<h_2$ and any $f \in W^{1, p}\left(B_R\right)$,
it holds that
$$
\Bigg( \frac{\big|B_R \cap\{f \leq h_1\} \big|}{|B_R|}\cdot
\frac{\big| B_R \cap\{f \geq h_2\}\big|}{|B_R|} \Bigg)^{\frac{N-1}{N}}
\leq \frac{CR^{1-\frac{N}{p}} }{h_2-h_1}\|\nabla f\|_{L^p\left(B_R\right)}
\Bigg(\frac{\big|B_R \cap\{h_1<f<h_2\}\big|}{|B_R|}\Bigg)^{\frac{p-1}{p}},
$$
for some constant $C\ge 1$ depending only on $N$ and $p$.
\end{lemma}

We now recall a classical iteration lemma. 

\begin{lemma}{\rm (\cite[Lemma 4.3]{HS}\label{lem-iteration})}
Let $\left\{Y_{j}\right\}_{j\in \Bbb{N}}$ be a sequence of positive
numbers, satisfying the recursive inequalities
\begin{align*}
Y_{j+1} \leq K b^{j} (Y_{j}^{1+\delta_1}+Y_{j}^{1+\delta_2}), ~~ j=0,1,2,
\ldots,
\end{align*}
where $K>0, b>1$ and $\delta_2\ge\delta_1>0$ are given numbers. If
$$
Y_{0} \leq \min \left\{ 1,(2 K)^{-\frac{1}{\delta_{1}}} b^{-\frac{1}{\delta_{1}^{2}}}\right\}
\quad \text{or}\quad
Y_{0} \leq \min \left\{(2 K)^{-\frac{1}{\delta_{1}}} b^{-\frac{1}{\delta_{1}^{2}}},(2 K)^{-\frac{1}{\delta_{2}}} b^{-\frac{1}{\delta_{1} \delta_{2}}-\frac{\delta_{2}-\delta_{1}}{\delta_{2}^{2}}}\right\},
$$
then $Y_{j} \leq 1$ for some $j\in \mathbb{N}$. Moreover,
$$
Y_{j} \leq \min  \left\{1,(2 K)^{-\frac{1}{\delta_{1}}} b^{-\frac{1}{\delta_{1}^{2}}} b^{-\frac{j}{\delta_{1}}}\right\} ~~ \text { for all } j \geq j_{0},
$$
where $j_{0}$ is the smallest $j\in \mathbb{N} \cup\{0\}$ satisfying $Y_{j} \leq 1$. In particular,
$Y_{j}$ converges to zero as $j\rightarrow\infty$.
\end{lemma}

We end the subsection with a technical lemma that plays an important role in the forthcoming context.

\begin{lemma}\label{lem-a}
Let $p,q,s$ satisfy \eqref{exp} and \eqref{pqs}.
Then there is a positive constant $C=C(N, p, q, s)$ such that for each $f \in W^{1, p}\left(B_r\right)$ and any $L_0>0$, we have
\begin{align*}
\dashint_{B_r}\left|\frac{f}{r}\right|^p+L_0 \left|\frac{f}{r^ s}\right|^q d x
& \leq C L_0 r^{(1-s) q}
\left(\dashint_{B_r}|D f|^p d x\right)^{\frac{q}{p}}+C \left(\frac{|\operatorname{supp} f|}{\left|B_r\right|}\right)^{\sigma}\dashint_{B_r}|D f|^p d x \nonumber\\
& +C\left(\frac{|\operatorname{supp} f|}{\left|B_r\right|}\right)^{p-1} \dashint_{B_r}\left|\frac{f}{r}\right|^p+L_0\left|\frac{f}{r^s}\right|^q d x
\end{align*}
with
$$
\sigma:= \begin{cases}\frac{p}{N} & \text { for } 1<p<N,
\\ 1-\frac{p}{q} & \text { for } p \geq N.\end{cases}
$$
\end{lemma}
\begin{proof}
[\bf {Proof.}]
 Applying the H\"{o}lder inequality and the Sobolev-Poincar\'{e} inequality, we can see
$$
\begin{aligned}
\dashint_{B_r}\left|\frac{f}{r^s}\right|^q\, d x
& \leq C \dashint_{B_r}\left|\frac{f-(f)_{B_r}}{r^s}\right|^q d x+C\left|\frac{(f)_{B r}}{r^s}\right|^q \\
& \leq C\left(\dashint_{B_r}\left|\frac{f-(f)_{B_r}}{r^s}\right|^{p_q^*
} \,d x\right)^{\frac{q}{p_q^*
}}
+C\left|\frac{(f)_{B_r}}{r^s}\right|^q \\
& \leq C r^{(1-s) q}\left(\dashint_{B_r}|D f|^p \,d x\right)^{\frac{q}{p}}+C\left|\frac{(f)_{B_r}}{r^s}\right|^q,
\end{aligned}
$$
where
$$
p_q^*= \begin{cases}p^* & \text { for } 1<p<N,
\\ q& \text { for } p \geq N.\end{cases}
$$
In a similar way, we get
$$
\begin{aligned}
\dashint_{B_r}\left|\frac{f}{r}\right|^p d x
& \leq C \dashint_{B_r}\left|\frac{f-(f)_{B_r}}{r}\right|^p \chi_{\{f\neq0\}}\,dx+C\left|\frac{(f)_{B r}}{r}\right|^p \\
& \leq C\left(\frac{|\operatorname{supp} f|}{\left|B_r\right|}\right)^{\sigma}
\left(\dashint_{B_r}\left|\frac{f-(f)_{B_r}}{r}\right|^{p_q^*} d x\right)^{\frac{p}{p_q^*}}+C\left|\frac{(f)_{B_r}}{r}\right|^p \\
& \leq C \left(\frac{|\operatorname{supp} f|}{\left|B_r\right|}\right)^{\sigma}\dashint_{B_r}|D f|^p d x
+C\left|\frac{(f)_{B_r}}{r}\right|^p.
\end{aligned}
$$
It follows from the H\"{o}lder inequality that
\begin{align*}
L_0\left|\frac{(f)_{B_r}}{r^s}\right|^q+\left|\frac{(f)_{B_r}}{r}\right|^p
 & \leq L_0 r^{-s q}\left(\frac{|\operatorname{supp} f|}{\left|B_r\right|}\right)^{q-1} \dashint_{B_r}|f|^q d x
 +r^{-p}\left(\frac{|\operatorname{supp} f|}{\left|B_r\right|}\right)^{p-1} \int_{B_r}|f|^p d x \nonumber\\
& \leq\left(\frac{|\operatorname{supp} f|}{\left|B_r\right|}\right)^{p-1} \dashint\left|\frac{f}{r}\right|^p+L_0\left|\frac{f}{r^s}\right|^q\,dx.
\end{align*}
A combination of the above estimates results in the claim.
\end{proof}


\section{Caccioppoli estimates}

This section is devoted to establishing the Caccioppoli estimates that encodes all the information needed to show (local) regularity of minimizers including boundedness, H\"{o}lder continuity and Harnack estimates.

\begin{lemma} [\label{lemCa}Caccioppoli inequality]
Let $B_{2R}\equiv B_{2R}(x_0)\subset\subset\Omega$. Suppose that $u\in \mathcal{A}(\Omega)$ is a minimizer to \eqref{q1}. Then there exists a positive constant $C$ only depending on $s,p,q,\Lambda,N$ such that for any $0<\rho<r\le R$,
\begin{align*}
&\quad \int_{B_{\rho}} \left|D w_{\pm}\right|^p d x+\int_{B_{\rho}} \int_{B_{\rho}}  a(x,y)\left|w_{\pm}(x)-w_{\pm}(y)\right|^q d\mu+
\int_{B_{\rho}}w_{\pm}(x) \Big(\int_{ \Bbb{R}^N }\frac{a(x,y)w^{q-1}_{\mp}(y)}{|x-y|^{N+sq}} d y\Big)  d x\nonumber\\
&\leq \frac{C}{ (r-\rho)^{p}  }\int_{B_{r}}  w_{\pm} ^p d x
+\frac{C}{(r-\rho)^q}\int_{B_r}\int_{B_{r}}a(x,y)\frac{|w_\pm(x)+w_\pm(y)|^q}{|x-y|^{N+(s-1)q} }\,dxdy \nonumber\\
&\quad+
\frac{Cr^{N+s q}}{(r-\rho)^{N+s q}}
\int_{B_r} \int_{\mathbb{R}^N \backslash  B_{\rho}}a(x,y)\frac{w_{ \pm}^{q-1}(x) w_{ \pm}(y)}{\left|x-x_0\right|^{N+s q}} d x d y
+C\int_{B_{r}}|f|w_\pm\,dx,
\end{align*}
where $w_\pm=(u-k)_{\pm}$ with a level $k\in\Bbb{R}$.
\end{lemma}
\begin{proof}[\bf Proof.]
We only give the sketch of proof, because that is similar to \cite[Proposition 7.5]{Co}. Choose any $\rho_1,r_1$ satisfying $\rho\le \rho_1<r_1\le r$. Let $\phi \in C_0^{\infty}\left(B_{(r_1+\rho_1) / 2}\right)$ be a cut-off function satisfying $0 \leq \phi \leq 1$, $\phi\equiv 1$ in $B_{\rho_1}$ and $|\nabla \phi| \leq 4 /(r_1-\rho_1)$. Using $v = u-w_+\phi$ as a test function in \eqref{q1}, we can see
\begin{align}\label{c1}
& 0\leq \int_{\Omega}\big(F(x,D v)-F(x,D u)\big) d x+ \int_{B_{r_1}}\int_{B_{r_1}}a(x,y)
\big(|v(x)-v(y)|^q-|u(x)-u(y)|^q\big) d \mu  \nonumber\\
&\quad+2 \int_{\mathbb{R}^{N} \backslash B_{r_1}} \int_{B_{r_1}}a(x,y)
\big(|v(x)-v(y)|^q-|u(x)-u(y)|^q\big)  d \mu+\int_{B_{r_1}}(v-u)f(x) d x \nonumber\\
&=:I_{1}+I_{2}+I_{3}+I_{4}.
\end{align}
We estimate terms $I_1$, $I_2$, $I_3$ and $I_4$ in \eqref{c1} separately. For $I_1$, we obtain
\begin{align*}
I_1\leq C \int_{B_{r_1} \backslash B_{\rho_1}}\left|D w_{+}\right|^p d x+C \int_{B_{r}} \left|\frac{w_{+}}{r_1-\rho_1}\right|^p d x-\Lambda^{-1} \int_{B_{r_1}} \left|D w_{+}\right|^p d x.
\end{align*}
The details can be found in \cite[Page 7]{By22}. Following the computations in \cite[P4819--4821]{Co}, we know that
\begin{align*}
I_2+I_3& \leq-\frac{1}{C} \int_{B_{\rho_1}} \int_{B_{\rho_1}}a(x,y)\left|w_{+}(x)-w_{+}(y)\right|^q d \mu  \nonumber\\
&\quad-\frac{1}{C} \int_{B_{\rho_1}} w_{+}(x)\left(\int_{\Bbb{R}^N} \frac{a(x,y)w^{q-1}_{-}(y)}{|x-y|^{N+s q}} d y\right) d x   \nonumber\\
&\quad+C\iint_{(B_{r_1}\times B_{r_1})\backslash (B_{\rho_1}\times B_{\rho_1})}a(x,y)\left|w_{+}(x)-w_{+}(y)\right|^q d \mu   \nonumber\\
&\quad+\frac{C}{(r_1-\rho_1)^q}\int_{B_{r_1}}\int_{B_{r_1}}a(x,y)\frac{|w_+(x)+w_+(y)|^q}{|x-y|^{N+(s-1)q}}\,dxdy   \nonumber\\
&\quad+\frac{C r^{N+s q}}{(r_1-\rho_1)^{N+s q}}
 \int_{B_{r}} \int_{\mathbb{R}^N \backslash  B_{\rho}}\frac{  a(x,y)w_{+}^{q-1}(x) w_{+}(y)}{\left|x-x_0\right|^{N+s q}} d x d y
\end{align*}
with some universal $C\ge1$. For the term $I_4$, it is easy to see
\begin{align*}
(v-u)f(x)=\phi(x)w_+(x)f(x),
\end{align*}
which directly gives the estimate of $I_4$ as below,
\begin{align*}
I_4\le \int_{B_{r}}w_+(x)|f(x)|dx.
\end{align*}

We can conclude from the estimates of $I_1$--$I_4$ that
\begin{align}\label{i1-i4}
 &\int_{B_{\rho_1}} \int_{B_{\rho_1}} a(x,y) \left|w_{+}(x)-w_{+}(y)\right|^q d\mu
+\int_{B_{\rho_1}} \left|D w_{+}\right|^p d x \nonumber\\
&\quad +  \int_{B_{\rho_1}} w_{+}(x)\left(\int_{ \Bbb{R}^N } \frac{ a(x,y)w^{q-1}_{-}(y) }{|x-y|^{N+s q}} d y\right) d x\nonumber\\
\leq & C\bigg(\int_{B_{r_1} \backslash B_{\rho_1}} \int_{B_{r_1} \backslash B_{\rho_1}}a(x,y) \left|w_{+}(x)-w_{+}(y)\right|^q d\mu
+\int_{B_{r_1} \backslash B_{\rho_1}} \left|D w_{+}\right|^p d x\nonumber\\
& \quad+\frac{1}{(r_1-\rho_1)^q}\int_{B_{r_1}}\int_{B_{r_1}}a(x,y)\frac{|w_+(x)+w_+(y)|^q}{|x-y|^{N+(s-1)q}}\,dxdy 
+\frac{1}{(r_1-\rho_1)^p}\int_{B_{r}}w_{+}^p d x   \nonumber\\
& \quad+ \frac{r^{N+s q}}{(r_1-\rho_1)^{N+s q}}
 \int_{B_r} \int_{\mathbb{R}^N \backslash  B_{\rho}}\frac{  a(x,y)w_{+}^{q-1}(x) w_{+}(y)}{\left|x-x_0\right|^{N+s q}} d x d y+\int_{B_{r}}w_+(x)|f(x)|dx\bigg).
\end{align}
Let us define $\Phi(t)$ as below,
$$
\begin{aligned}
\Phi(t)= & \int_{B_{t}} \int_{B_{t}} a(x,y) \left|w_{+}(x)-w_{+}(y)\right|^q d\mu
+\int_{B_{t}} \left|D w_{+}\right|^p d x \nonumber\\
&\quad +  \int_{B_{t}} w_{+}(x)\left(\int_{  \Bbb{R}^N } \frac{a(x,y)w^{q-1}_{-}(y)}{|x-y|^{N+s q}} d y\right) d x,\qquad t>0.
\end{aligned}
$$
Then it follows by \eqref{i1-i4} that
$$
\begin{aligned}
\Phi\left(\rho_1\right) & \leq C\big(\Phi\left(r_1\right)-\Phi\left(\rho_1\right)\big)
 +\frac{C}{(r_1-\rho_1)^q}\int_{B_{r_1}}\int_{B_{r_1}}a(x,y)\frac{|w_+(x)+w_+(y)|^q}{|x-y|^{N+(s-1)q}}\,dxdy\\
& \quad+\frac{C}{(r_1-\rho_1)^p}\int_{B_{r}}w_{+}^p d x + \frac{Cr^{N+s q}}{(r_1-\rho_1)^{N+s q}}
 \int_{B_r} \int_{\mathbb{R}^N \backslash  B_{\rho}}\frac{  a(x,y)w_{+}^{q-1}(x) w_{+}(y)}{\left|x-x_0\right|^{N+s q}} d x d y\\
 & \quad+C\int_{B_{r}}w_+(x)|f(x)|dx
\end{aligned}
$$
with $C=C(N, s, p, q, \Lambda)$.
This allows us to apply the the technical lemma \cite[Lemma 2.5]{Fi22} to arrive at the desired estimate.
\end{proof}

\section{Local boundedness}
This section is devoted to obtaining, by Caccioppoli inequality, the local boundedness of the minimizers to \eqref{q1}, Theorems \ref{th1} and \ref{th2}.

For arbitrarily fixed center $x_0\in\Omega$ and radius $r\in(0,1)$ satisfying $B_{2r}\equiv B_{2r}(x_0)\subset\subset \Omega$,
let us take a decreasing sequence
\begin{align}\label{r}
r_{i}=2^{-1} r+2^{-i-1}r, \quad i=0,1,2,\ldots.
\end{align}
The balls $B_i$ are chosen as 
\begin{align}\label{Bj}
B_{i}=B_{r_i}(x_0), \quad i=0,1,2,\ldots.
\end{align}

We define sequences of increasing levels and the corresponding functions as below:
\begin{align}\label{w}
&k_{i}=\left(1-2^{-i}\right) \bar{k} \quad  \text{with } \bar{k}>0, \quad w_{i}=\left(u-k_{i}\right)_{+},
 ~~ i=0,1,2,
\ldots.
\end{align}
Note that
$$r_i-r_{i+1}=2^{-i-2} r, \quad k_{i+1}-k_i=2^{-i-1} \bar{k} \quad \text {and}\quad  w_{i+1} \leq w_i,$$
which shall be used many times in the sequel.
In the two coming lemmas, we deal with the Caccioppoli
inequality written for the function $w_{i}$
over the domains $B_{i+1}$ and $B_i$.

\begin{lemma}\label{lem3.2}
Suppose that $u$ is a minimizer to \eqref{q1} with $K_{sq}$ and $F$ satisfying \eqref{F}--\eqref{ca} with $a_0>0$, and the function $f|_{B_r}$ belongs to $L^{\gamma}(B_{r})$
with $\gamma>\frac{q}{q-1}$. 
Let the notations $B_i$ and $w_i$ be given in \eqref{Bj}--\eqref{w} respectively. Then we have for
all $i\in \Bbb{N}$ that
\begin{align*}
\dashint_{B_{i+1}}H_{r}(w_{i+1}) dx
& \leq \frac{C 2^{i(N+q)}  }{\big(H_{r}\left(k_{i+1}-k_i\right)\big)^{\frac{1}{\kappa^{\prime}}}}
\left(\dashint_{B_i} H_{r}(w_{i}) d x\right)^{1+\frac{1}{\kappa^{\prime}}}
\Big(1+\frac{\overline{\operatorname{Tail}} (u_+;x_0,r/2)}{h_{r}( k_{i+1}-k_i)}\Big)
\nonumber\\
&\quad+ \frac{C 2^{i(N+q)}  }{\big(H_{r}\left(k_{i+1}-k_i\right)\big)^{\frac{1}{\kappa^{\prime}}}}
\left(\dashint_{B_i} H_{r}(w_{i}) d x\right)^{\frac{1}{\gamma'}+\frac{1}{\kappa^{\prime}}}
\frac{ d  }{g_{r}( k_{i+1}-k_i)},
\end{align*}
where $\kappa:=p^*/p$, $\frac{1}{\kappa}+\frac{1}{\kappa'}=1$, $d:=\Big(\dashint_{B_{r}}|f(x)|^\gamma \,dx\Big)^{\frac{1}{\gamma}}$,  the functions $H_r(\cdot), g_r(\cdot)$ are determined by \eqref{ha}--\eqref{Ha}, and
$$
\overline{\operatorname{Tail}} (u_+;x_0,r/2):=\int_{\mathbb{R}^N \backslash  B_{r/2}(x_0)}  \frac{u_+^{q-1}(x)}{\left|x-x_0\right|^{N+s q}} \,d x d y.
$$
\end{lemma}
\begin{proof}
[\bf Proof.]
By utilizing the (fractional) Sobolev embedding theorem and the H\"{o}lder inequality, we have
\begin{align}\label{sup1}
&\quad\dashint_{B_{i+1}} H_{r}(w_{i+1}) dx  \nonumber\\
& \leq \left(\frac{\left| A^+(k_{i+1}, w_{i+1})\right|}{\left|B_{i+1 }\right|}\right)^{\frac{1}{\kappa'}}
 \bigg( \dashint_{B_{i+1}}\Big(\frac{w_{i+1}^p}{ r_{i+1}^p}+\frac{ w_{i+1}^q }{r_{i+1}^{sq}}\Big)^\kappa dx \bigg)^{\frac{1}{\kappa }}
\nonumber\\
&\leq C \left(\frac{\left| A^+(k_{i+1}, w_{i+1})\right|}{\left|B_{i+1 }\right|}\right)^{\frac{1}{\kappa'}}
\bigg(  \dashint_{B_{i+1} } \int_{B_{i+1}}\left|w_{i+1}(x) -w_{i+1}(y) \right|^q d \mu
+\dashint_{ B_{i+1} } |Dw_{i+1}|^p d x\bigg) \nonumber\\
&\quad +C \left(\frac{\left| A^+(k_{i+1}, w_{i+1})\right|}{\left|B_{i+1 }\right|}\right)^{\frac{1}{\kappa'}}
\dashint_{B_{i+1}} H_{r}(w_{i+1}) dx,
\end{align}
where we used the fact $\kappa:=\min\{p^*/p,q^*_s/q\}$ and $\frac{1}{\kappa}+\frac{1}{\kappa'}=1$.
It is not hard to verify that
$$
\left| A^+(k_{i+1}, r_{i+1})\right| \le
\int_{A^+(k_{i+1}, r_{i+1})} \frac{H_{r}((u-k_{i})_+)}{H_{r}( k_{i+1}-k_i)} dx \le
\int_{B_{i}} \frac{H_{r}(w_{i})}{H_{r}( k_{i+1}-k_i)} dx,
$$
where the set $A^+$ is defined as \eqref{A}. Applying Lemma \ref{lemCa} with $a_0>0$ and \eqref{r}, we get
\begin{align}\label{sup2}
&\quad\int_{B_{i+1}}\dashint_{B_{i+1}}\frac{\left|w_{i}(x)-w_{i}(y)\right|^{q}}{|x-y|^{N+sq}}
d x d y
+ \dashint_{B_{i+1}}|D {w}_{i} |^p  dx
\nonumber\\
&\leq C\left(\frac{ r_i}{r_i-r_{i+1}}\right)^{N+q}
\bigg( \dashint_{B_i}\frac{w_{i+1}^p}{ r^p}+\frac{w_{i+1}^q }{r^{sq}} dx+
\dashint_{B_i} \int_{\mathbb{R}^N \backslash  B_{i+1}}  \frac{w_{i+1}^{q-1}(x)w_{i+1}(y)}{\left|x-x_0\right|^{N+s q}} d x d y
+\dashint_{B_{i}}|f|w_{i+1}\,dx\bigg)\nonumber\\
&\leq C2^{(N+q)i}
\bigg( \dashint_{B_i} H_{r}(w_{i+1})dx
+\overline{\operatorname{Tail}} (u_+;x_0,r/2)\dashint_{B_i} w_{i+1}dx
+\dashint_{B_{i}}|f|w_{i+1}\,dx\bigg).
\end{align}
Now we analyse
$$
\dashint_{B_i} w_{i+1}\,dx
\le\dashint_{B_i}\frac{h_{r}((u-k_{i})_+)}{h_{r}( k_{i+1}-k_i)} (u-k_{i+1})_+\,dx
\le \frac{1}{h_{r}( k_{i+1}-k_i)}\dashint_{B_i}H_{r}((u-k_{i})_+)\,dx
$$
and
$$
w_{i+1}=(u-k_{i+1})_+\le (u-k_{i+1})_+\left[\frac{(u-k_{i})_+}{k_{i+1}-k_i}\right]^{l-1}\le \frac{(u-k_{i})_+^l}{(k_{i+1}-k_i)^{l-1}},\quad \text{for } l\ge1.
$$
For the last integral in \eqref{sup2}, we have by the H\"{o}lder inequality that
\begin{align}\label{sup3}
\dashint_{B_{i}}|f|w_{i+1}\,dx
&\leq \frac{1}{\left(k_{i+1}-k_i\right)^{\bar{q}-1}}
\left(\dashint_{B_{i}}|f|^\gamma \,d x\right)^{\frac{1}{\gamma}}\left(\dashint_{B_{i}} w_i^{\bar{q} \cdot \gamma^{\prime}} \,dx\right)^{\frac{1}{\gamma^{\prime}}} \nonumber\\
& =\frac{r^{\bar{q}}}{\left(k_{i+1}-k_i\right)^{\bar{q}-1}}
\left(\dashint_{B_{i}}|f|^\gamma \,d x\right)^{\frac{1}{\gamma}}
\left(\dashint_{B_{i}}\frac{w_i^q}{r^q} \,d x\right)^{\frac{1}{\gamma^{\prime}}} \nonumber\\
& \leq\frac{r^{\bar{q}}}{\left(k_{i+1}-k_i\right)^{\bar{q}-1}}
\left(\dashint_{B_{i}}|f|^\gamma \,d x\right)^{\frac{1}{\gamma}}
\left(\dashint_{B_{i}} H_{r}(w_i)\,d x\right)^{\frac{1}{\gamma^{\prime}}}
\end{align}
with $\bar{q}:=q/\gamma'=(\gamma-1)q/\gamma $. Substituting \eqref{sup2} and \eqref{sup3}
into \eqref{sup1} leads to the desired estimate.
\end{proof}

\begin{lemma}\label{lem3.3}
Suppose that $u$ is a minimizer to \eqref{q1} with $K_{sq}$ and $F$ satisfying \eqref{F}--\eqref{ca} with $a_0>0$, and 
the function $f|_{B_r}$ belongs to $L^{\gamma}(B_{r})$ with
\begin{align*}
\gamma>\max\Big\{\frac{N}{p},\frac{q}{q-1}\Big\}.   
\end{align*}
Then, for any $\delta\in(0,1)$, we have
\begin{align*}
\sup _{B_{r/2}} u_{+}
\leq C_\delta H_{r}^{-1}\left(\dashint_{B_{ r}} H_{r}\left(u_{+}\right) \,d x\right)
+\delta h_{r}^{-1}\big(r^{-sq}\operatorname{Tail}^{q-1} (u_+;x_0,r/2)\big)
+\delta g_{r}^{-1}(d),
\end{align*}
 where $d=\Big(\dashint_{B_{r}}|f(x)|^\gamma \,dx\Big)^{\frac{1}{\gamma}}$  and the constant $C_{\delta}>0$ depends on $s,p,q,N,\Lambda,a_0,A_0$ and $\delta$.
 \end{lemma}
\begin{proof}
[\bf Proof.]
Recalling the definitions of $w_i,B_i$ and $H_{r}(\cdot)$, we denote
$$
Y_i=\dashint_{B_i} H_{r}(w_{i}) \,d x.
$$
Lemma \ref{lem3.2} gives the following estimate for $Y_i$:
\begin{align*}
Y_{i+1} & \le\frac{C 2^{i(N+q)}}{\left(H_{r}(2^{-i-1}\bar{ k})\right)^{1 / \kappa^{'}}}
\left(1+\frac{\operatorname{Tail}^{q-1}(u_+;x_0,r/2)}{ r^{sq}h_{r}(2^{-i-1} \bar{k})}\right) Y_i^{1+\frac{1}{\kappa^{\prime}}} \nonumber\\
& \quad+\frac{C 2^{i(N+q)}}{\left(H_{r}(2^{-i-1} \bar{k})\right)^{1 / \kappa^{\prime}}}
\frac{d}{g_{r}(2^{-i-1} \bar{k})} Y_i^{1+\left(\frac{1}{\kappa^{\prime}}+\frac{1}{\gamma^{\prime}}-1\right)}.
\end{align*}
By observing
$$
H_{r}(2^{-i-1} \bar{k}) \geq 2^{-(i+1) q}H_{r}( \bar{k}), \ h_{r}(2^{-i-1} \bar{k}) \geq 2^{-(i+1)(q-1)} h_{r}(\bar{k})
$$
 and
$$g_{r}(2^{-i-1} \bar{k})=2^{-(i+1)(\bar{q}-1)} g_{r}(\bar{k}),$$
then we get
\begin{align*}
Y_{i+1} &\le\frac{C 2^{i(N+2q+\frac{q}{\kappa'})}}{H_{r}( \bar{k})^{\frac{1}{\kappa'}   }  }
\left(1+\frac{\operatorname{Tail}^{q-1} (u_+;x_0,r/2)}{ r^{sq}h_{r}\left( \bar{k}\right)}\right) Y_i^{1+\frac{1}{\kappa^{\prime}}} \nonumber\\
& \quad+\frac{C 2^{i(N+q+\bar{q}+\frac{q}{\kappa'})}}{H_{r}( \bar{k})^{\frac{1}{\kappa'}   }  }
\frac{d}{g_{r}( \bar{k} )} Y_i^{1+\left(\frac{1}{\kappa' }+\frac{1}{\gamma'  }-1\right)}.
\end{align*}
Now we pick first $\bar{k}  \geq \delta h_{r}^{-1}\left(r^{-sq} \operatorname{Tail}^{q-1} (u_+;x_0,r/2)\right)
+\delta g_{r}^{-1}(d)$ such that
$$
\frac{\delta ^{q} \operatorname{Tail}^{q-1} (u_+;x_0,r/2)}{r^{sq}h_{r}\left( \bar{k}\right)}
\leq \frac{\operatorname{Tail}^{q-1}(u_+;x_0,r/2)}{ r^{sq}h_{r}\left(  \bar{k} / \delta\right)}\leq 1
$$
and
$$
\frac{\delta^{q}d }{g_{r}(\bar{k})}
\leq
\frac{d }{g_{r}\left( \bar{k} / \delta\right)}
\leq 1.
$$
Therefore, we have
\begin{align*}
Y_{i+1} \le
\frac{C 2^{i\theta}}{\delta^{q}  H_{r}( \bar{k})^{ \frac{1}{\kappa' } }}
\Big(Y_i^{1+\frac{1}{\kappa^{\prime}}} +Y_i^{1+\left(\frac{1}{\kappa^{\prime}}+\frac{1}{\gamma^{\prime}}-1\right)}\Big),
\end{align*}
where $\theta:=N+2q+\frac{q}{\kappa'}$. Moreover, we can verify
$$
\frac{1}{\kappa'}+\frac{1}{\gamma'}=\frac{\kappa-1}{\kappa}+\frac{\gamma-1}{\gamma} >1
$$
due to $\gamma>\frac{\kappa}{\kappa-1}$ ensured by \eqref{gam}.

With taking $\sigma:=\frac{\kappa'}{1 / \kappa'+1 /\gamma^{\prime}-1}+\left(\kappa'\right)^2(1-\frac{1}{\gamma^{\prime}})$,
we aim at selecting $\bar{k}$ large enough such that
\begin{align}\label{1sup3}
Y_0&=\dashint_{B_{r}}H_{r}\left(\left(u-\bar{k}\right)\right)\, d x
 \leq\dashint_{B_{r}} H_{r}\left(u_{+}\right)\, d x
 \leq\Bigg(\frac{2C}{\delta^{q} \left(H_{r}\left(\bar{ k}\right)\right)^{1 / \kappa^{\prime}}  }\Bigg)^{-\kappa'}
 2^{-\theta\sigma}.
\end{align}
By noticing
$$\Bigg(\frac{2C}{\delta^{q}  \left(H_{r}\left(\bar{ k}\right)\right)^{1 / \kappa^{\prime}}  }\Bigg)^{-\kappa'}
 2^{-\theta\sigma}
 =2^{-\theta\sigma}\Bigg(\frac{\delta^{q} }{ 2C  }\Bigg)^{\kappa'}H_{r}\left( \bar{k} \right) $$
and doing some calculations, we select $\bar{k}$ such that
$$\bar{k}\ge H_{r}^{-1}\left(2^{\theta\sigma}\left(\frac{2C}{\delta^{q}  }\right)^{\kappa'}
\dashint_{B_{r}} H_{r}(u_{+}) \,d x\right) $$
to ensure the validity of \eqref{1sup3}.
We eventually choose
$$
\bar{k}=H_{r}^{-1}\left(2^{\theta\sigma}\left(\frac{2C}{\delta^{q} }\right)^{\kappa'}\dashint_{B_{r}} H_{r}(u_{+}) \,d x\right)
+\delta h_{r}^{-1}\big(\operatorname{Tail}^{q-1} (u_+;x_0,r/2)/r^{sq}\big)
+\delta g_{r}^{-1}(d).
$$
At this moment, we can exploit Lemma \ref{lem-iteration} to conclude that $Y_i \rightarrow 0$ as $i \rightarrow \infty$.
This ends the proof.
\end{proof}

\noindent{\bf Proof of Theorem \ref{th1}.}
The theorem is a direct corollary of Lemma \ref{lem3.3}.
\qquad$\Box$

\medskip

Next, we consider the more general scenario that $a_0=0$ in \eqref{ca}, which enjoys indeed the properties of double phase problems at this time.

\begin{lemma}\label{lem3.1}
Let $u$ be a minimizer to \eqref{q1}. 
Suppose that \eqref{F}--\eqref{exp} with $a_0=0$ are in force. Let
\begin{align*}
\left\{\begin{array}{l}
q \leq\frac{Np}{N-p}=: p^* \text{  for  } 1<p\le N,\\[2mm]
q<\infty\text{  for  } p> N
\end{array}\right.
\end{align*}
hold true. The function $f|_{B_r}$ belongs to $L^{\gamma}(B_{r})$ with $\gamma>\frac{p}{p-1}$.
Let the notations $B_i,w_i$ be given in \eqref{Bj}--\eqref{w}, respectively.
 Then we have for all $i\in \Bbb{N}$ that
\begin{align}\label{v0'}
r^{-p}\dashint_{B_{i+1}}H(w_{i+1}) d x
&\leq C r^{(1-s)q}2^{i m}\bigg(\frac{ 1 }{r^{sq}}+\frac{ 1}{ r ^p  }
+\frac{  \operatorname{Tail}_a^{q-1}\left(u_{+}, x_0, r / 2\right)}{r^{sq}\bar{k}^{q-1}} \bigg)^{   \frac{q}{p}}
\left(\dashint_{B_i} H(w_i) dx\right)^{   \frac{q}{p}}    \nonumber\\
&\quad+\frac{Cr^{(1-s)q}2^{i m}d^{   \frac{q}{p}}  }{\bar{k}^{(\bar{p}-1)q/p}}
\left(\dashint_{B_i}  H(w_i)d x\right)^{\frac{q(\gamma-1)}{p\gamma}}     \nonumber\\
&\quad+ C2^{i m}\bigg(\frac{ 1 }{r^{sq}}+\frac{1}{ r ^p  }
+\frac{ \operatorname{Tail}_a^{q-1}\left(u_{+}, x_0, r / 2\right)}{r^{sq}\bar{k}^{q-1}} \bigg)
\Big(\dashint_{B_i} H(w_i) dx\Big)^{2-   \frac{p}{p^*}}    \nonumber\\
&\quad+\frac{ C 2^{im} d }{   \bar{k}^{q\left(1-p/p^*\right)+\bar{p}-1}  }
\left(\dashint_{B_i}  H(w_i)d x\right)^{ 2-   \frac{p}{p^*}-\frac{1}{\gamma} }   \nonumber\\
&\quad+\frac{ C2^{i m}}{r^{s q}\bar{k}^{q(p-1)}  }\left(\dashint_{B_i}  H(w_i)d x\right)^{p},
\end{align}
where $\bar{p}:=(\gamma-1)p/\gamma$, $m:=\max\{(N+s q)q/p,q+\bar{p}-p q / p^*, q(p-1)\}$, $d=\Big(\dashint_{B_{r}}|f(x)|^\gamma\,dx\Big)^{\frac{1}{\gamma}}$ and
the function $H(\cdot)$ is defined by $H(t):=t^p+a^+_r t^q$ for $t>0$.
Here $\operatorname{Tail}_a(u_+,x_0,r/2)$ is given by \eqref{taila}.
\end{lemma}

\begin{proof}[\bf Proof.]
We choose $\rho=r_{i+1}$ and $r=r_i$ in Lemma \ref{lemCa} to get that
 for any $\beta$ satisfying $1\le\beta< p$,
\begin{align}\label{v1}
\dashint_{B_{i}}|D w_i |^p  \,dx
&\leq \frac{C}{ (r_i-r_{i+1})^{p}  }\dashint_{B_{i}}  w_{i+1} ^p d x
+\frac{C}{(r_i-r_{i+1})^q}\dashint_{B_i}\int_{B_i}a(x,y)\frac{|w_{i+1}(x)+w_{i+1}(y)|^q}{|x-y|^{N+(s-1)q} }\,dxdy \nonumber\\
&\quad+\frac{Cr_i^{N+s q}}{(r_i-r_{i+1})^{N+sq}}
\dashint_{B_i} \int_{\mathbb{R}^N \backslash  B_{i+1}}a(x,y)\frac{w_{i+1}^{q-1}(x) w_{i+1}(y)}{\left|x-x_0\right|^{N+s q}} \,d x d y
+C\dashint_{B_{i}}|f|w_{i+1}\,dx    \nonumber\\
&\leq \frac{C2^{iq}a_r^+}{r^{sq}}\dashint_{B_i} w_i^q \,dx
+\frac{C2^{ip}}{r^p}\dashint_{B_{i}}w_i^p\,d x+\frac{C2^{(i+1)(\beta-1)}}{\bar{k}^{\bar{q}-1}}\dashint_{B_i}|f| w_i^{\beta}\,dx
\nonumber\\
&\quad+
\frac{C 2^{i(N+s q+q)} \operatorname{Tail}_a^{q-1}\left(u_{+}, x_0, r / 2\right)}{r^{sq}\bar{k}^{q-1}} \dashint_{B_i} w^q_i \,d x,
\end{align}
where we utilized $w_{i+1} \leq u_{+}$, $r\ge r_i \geq r / 2$ and
$$
w_{i+1}=\left(u-k_{i+1}\right)_{+} 
\le\frac{2^{(i+1)(p-1)}}{\bar{k}^{\beta-1}} w_i^\beta .
$$
Now we consider the integral $\dashint_{B_i}|f| w_i^{\beta} \,d x$ and use the H\"{o}lder inequality to find that
\begin{align}\label{v2}
\dashint_{B_i}|f| w_i^{\beta} \,d x
 & \leq\left(\dashint_{B_i}|f|^\gamma \,d x\right)^{\frac{1}{\gamma}}
 \left(\dashint_{B_i} w_i^{\frac{\gamma\beta}{\gamma-1}} \,d x\right)^{\frac{\gamma-1}{\gamma}}\nonumber \\
& \leq C\left(\dashint_{B_r}|f|^\gamma \,d x\right)^{\frac{1}{\gamma}}
\left(\dashint_{B_i} w_i^{p} \,d x\right)^{\frac{\gamma-1}{\gamma}},
\end{align}
where $\beta$ was taken as $\beta=\bar{p}$.
Thus, by inserting \eqref{v2} into \eqref{v1}, we derive that
\begin{align}\label{v3}
\dashint_{B_{i+1}} |D w_{i+1}|^p \,dx
&\leq \bigg(\frac{C 2^{iq} }{r^{sq}}+\frac{C2^{ip}}{ r ^p  }
+\frac{C 2^{i(N+s q)} \operatorname{Tail}_a^{q-1}\left(u_{+}, x_0, r / 2\right)}{  r^{sq}\bar{k}^{q-1}  } \bigg)
\dashint_{B_i} H(w_i) \,dx
\nonumber\\
&\quad+\frac{C 2^{i(\bar{p}-1)}d}{\bar{k}^{\bar{p}-1}}\left(\dashint_{B_i}  H(w_i)\,d x\right)^{\frac{\gamma-1}{\gamma}}.
\end{align}

On the other hand, from Lemma \ref{lem-a}, there holds that
\begin{align}\label{v4}
r_{i+1}^{-q}\dashint_{B_{i+1}}H(w_{i+1}) \,d x
&\le
\dashint_{B_{i+1}}\left(\frac{w_{i+1}}{ r_{i+1} }\right)^p
+ a_r^+\left(\frac{w_{i+1}}{r_{i+1}^ s}\right)^q \,d x\nonumber\\
& \leq C  r_{i+1}^{(1-s) q} a_r^+
\left(\dashint_{B_{i+1}}|Dw_{i+1}|^p \,d x\right)^{\frac{q}{p}} \nonumber\\
&\quad +C \left(\frac{|A^+(k_{i+1}, r_{i+1})|}{\left|B_{i+1}\right|}\right)^{\sigma }
\dashint_{B_{i+1}}|D w_{i+1}|^p \,d x \nonumber\\
& \quad
+C\left(\frac{|A^+(k_{i+1}, r_{i+1})|}{\left|B_{i+1}\right|}\right)^{p-1} \dashint_{B_{i+1}}H\left( \frac{w_{i+1}}{r_{i+1}} \right)\,dx
\end{align}
with
$$
\sigma:= \begin{cases}\frac{p}{N} & \text { for } 1<p<N,
\\ 1-\frac{p}{q} & \text { for } p \geq N.\end{cases}
$$
Here the choice of $r$ directly ensures that
$$\dashint_{B_{i+1}}H\left( \frac{w_{i+1}}{r_{i+1}} \right)  \,d x
\le Cr^{-sq}\dashint_{B_{i+1}}H\left( w_{i+1} \right)  \,d x.$$
It is not hard to verify that
\begin{align}\label{v5}
\frac{|A^+(k_{i+1}, r_{i+1})|}{\left|B_{i+1}\right|}
 &\le
\frac{1}{\left|B_{i+1}\right|}\int_{A^+(k_{i+1}, r_{i+1})} \frac{\left(u-k_i\right)_{+}^p}{\left(k_{i+1}-k_i\right)^p}\,dx\nonumber\\
&=\frac{ 2^{(i+1)(p-1)} }{\left|B_{i+1}\right|}
\int_{A^+(k_{i+1}, r_{i+1})} \frac{ w_i^p }{\bar{k}^{p-1}}\,dx   \nonumber\\
&\le\frac{2^{(i+1)p}}{\bar{k}^{p}}\dashint_{B_{i}} H(w_i)\,dx
\end{align}
due to \eqref{r}. A combination of \eqref{v3}--\eqref{v5} infers that
\begin{align*}
r^{-p}\dashint_{B_{i+1}}H(w_{i+1}) \,d x
&\leq C r^{(1-s)q}\bigg(\frac{ 2^{iq} }{r^{sq}}+\frac{ 2^{ip}}{ r ^p  }
+\frac{ 2^{i(N+sq)} \operatorname{Tail}_a^{q-1}\left(u_{+}, x_0, r / 2\right)}{ r^{sq}\bar{k}^{q-1}} \bigg)^{   \frac{q}{p}}
\Big(\dashint_{B_i} H(w_i) \,dx\Big)^{   \frac{q}{p}}
\nonumber\\
&\quad+Cr^{(1-s)q}\bigg(\frac{2^{(i+1)(\bar{p}-1)}d}{\bar{k}^{\bar{p}-1}}\bigg)^{   \frac{q}{p}}
\left(\dashint_{B_i}  H(w_i)\,d x\right)^{\frac{q(\gamma-1)}{p\gamma}}\\
&\quad+ C\bigg(\frac{ 2^{iq} }{r^{sq}}+\frac{2^{ip}}{ r ^p  }
+\frac{ 2^{i(N+s q)} \operatorname{Tail}_a^{q-1}\left(u_{+}, x_0, r / 2\right)}{r^{sq}\bar{k}^{q-1}} \bigg)
\Big(\dashint_{B_i} H(w_i) \,dx\Big)^{ 1+ \sigma  }
\nonumber\\
&\quad+\frac{ C 2^{i\left(q+\bar{p}-p q / p^*\right)} d }{   \bar{k}^{q\left(1-p / p^*\right)+\bar{p}-1}  }
\left(\dashint_{B_i}  H(w_i)\,d x\right)^{ 1+\sigma-\frac{1}{\gamma} }\\
&\quad+\frac{ C2^{i q(p-1)}}{r^{s q}\bar{k}^{q(p-1)}  }
\left(\dashint_{B_i}  H(w_i)\,d x\right)^{p}.
\end{align*}
Hence, \eqref{v0'} is an immediate result of the above inequality by arrangements.
\end{proof}

Now we are ready to give the proof of boundedness result on the mixed local and nonlocal double phase functionals.

\medskip

\noindent{\bf Proof of Theorem \ref{th2}.}
Let the assumptions of Theorem \ref{th2} hold.
Now we set
$$
Y_{i}=\dashint_{B_{i}}H(w_i)\,d x, \quad i=0,1,2, \ldots,
$$
where $B_i,w_i$ are given in \eqref{Bj} and \eqref{w}, and $H(\cdot)$ is as defined in Lemma \ref{lem3.1}.
Letting
$$\bar{k}\ge\max\{1,\operatorname{Tail}_a\left(u_{+}, x_0, r / 2\right)\}$$
to be determined later, we can deduce from Lemma \ref{lem3.1} that
\begin{align}\label{1ta1}
Y_{i+1}&\le C 2^{im}\Big(Y_i^{ \frac{q}{p}  }+Y_i^{ \frac{(\gamma-1)q}{\gamma p}}+Y_i^{ 1+ \sigma }
+Y_i^{1+\sigma-\frac{1}{\gamma}}+Y_i^{ p}\Big),
\end{align}
where $m:=\max\{(N+s q)q/p , q+\bar{p}-p q / p^*, q(p-1)\}$, and the constant $C$ depends also upon $r,d$.
Due to the assumption \eqref{gam}, it can be checked that
$$ \frac{(\gamma-1)q}{\gamma p}>1  \quad \text{and} \quad \sigma-\frac{1}{\gamma}>0.$$
Additionally, we can see that $2-\frac{p}{p^*}>1$. 
Because $H(u)\in L^1(B_r)$ from \eqref{pqs}, one can see that
 $$
 Y_0=\dashint_{B_{r}}H((u-\bar{k})_+)\,d x\rightarrow 0
 $$
as $\bar{k}\rightarrow+\infty$. As a result, we can proceed to select $\bar{k}$ so large that
$$
\cdots \le Y_i\le Y_{i-1}\le \cdots\le Y_0\le 1
$$
for $i=1,2,3,\cdots$. Hence we rearrange the display \eqref{1ta1} as
\begin{align}\label{1ta2}
Y_{i+1}&\le C 2^{im}Y_i^{1+\tau}.
\end{align}
Here
$$
\tau:=\min\left\{\frac{q}{p}-1,\frac{(\gamma-1)q}{\gamma p}-1,\sigma,\sigma-\frac{1}{\gamma},p-1\right\}>0.
$$

Finally, we choose such a large number $\bar{k}$ that 
$$
Y_{0} \leq C^{-\frac{1}{\tau}} 2^{-\frac{m}{\tau^{2}}},
$$
which combining with \eqref{1ta2} and Lemma \ref{lem-iteration} guarantees that
\begin{align}\label{th1-8}
Y_j\rightarrow 0 \quad \text{as } j\rightarrow\infty.
\end{align}
Under the above election of $\bar{k}$, \eqref{th1-8} guarantees that $u\le 2\bar{k}$ in $B_{r/2}$. We could infer $u\in L^\infty(B_{r/2})$ applying the analogous argument to $-u$.
\qquad$\Box$
\section{Local H\"{o}lder continuity}

In this section, 
we first conclude expansion of positivity for minimizers of \eqref{q1} (see Lemmas \ref{lem4.1}--\ref{lem4.2}), which plays a crucial role on establishing H\"{o}lder continuity and Harnack inequalities.

\begin{lemma}\label{lem4.1}
Assume that $K_{sq}$ and $F$ satisfy \eqref{F} and \eqref{K} with $a_0>0$.
Let $B_{4 R}:=B_{4R}(x_0) \subset\subset\Omega$ with $R \leq 1$.
Let $u \in \mathcal{A}(\Omega) \cap L_{s q}^{q-1}\left(\mathbb{R}^N\right)$
be a minimizer of \eqref{q1}, where the function $f|_{B_{4 R}}$ belongs to $L^{\gamma}(B_{4 R})$ with $\gamma>1$.
Suppose that $u \ge M \text{   in   } B_{4 R}$
and
\begin{align}\label{hc0}
\left|B_{2 R} \cap\{u-M\geq t\}\right| \geq \nu\left|B_{2 R}\right|
\end{align}
for some $\nu \in(0,1)$, $t>0$ and $M\in\Bbb{R}$. Then for any $\delta \in\left(0, \frac{1}{2^8}\right]$, if
\begin{align}\label{hc}
\| f\|_{L^\gamma\left(B_{4 R}\right)}\left|B_{4 R}\right|^{-\frac{1}{\gamma}}
+(4R)^{-sq}\operatorname{Tail}^{q-1}\left((u-M)_{-} ; x_0, 4 R\right)
\leq h_{4R}(\delta t)
\end{align}
with $h_{4 R}(\delta t):= \frac{(\delta t)^{p-1}}{(4 R)^p}+ \frac{(\delta t)^{q-1}}{(4 R)^{sq}}$ as defined by \eqref{ha},
then there holds that
$$
\left|B_{2 R} \cap\{u-M<2 \delta t\}\right| \leq
\frac{C }{\nu}\Big(\delta^{(q-1)/2}
+ |\log \delta|^{-\frac{N(p-1)}{(N-1) p}}\Big)\left|B_{2 R}\right|,
$$
where $C>0$, independent of $M,t$, only depends on $s,p,q,N,\Lambda,a_0$ and $A_0$.
\end{lemma}
\begin{proof}[\bf Proof.]
Let $\ell \ge\frac{\delta t}{2}$ and set $w=u-M$.
It is not difficult to verify that $w$ belongs to $\mathcal{A}(\Omega) \cap L_{s q}^{q-1}\left(\mathbb{R}^N\right)$
and is also a minimizer of $\mathcal{E}$.
With the help of Lemma \ref{lemCa} and \eqref{hc}, we can derive that
\begin{align}\label{h1}
&\quad[(w-\ell)_-]_{W^{1,p}(B_{2R})  }^p+a_0
\int_{B_{2R} }(w(x) -\ell)_{-}\Big(\int_{ B_{2R} }\frac{(w(y)-\ell)^{q-1}_{+}}{|x-y|^{N+sq}} d y\Big)  \,d x\nonumber\\
&\leq C
\bigg(
\int_{B_{4R} } \frac{(w-\ell)_-^q }{(4R)^{sq}} \,dx
+\int_{B_{4R}}  \frac{(w-\ell)_-^p}{ (4R)^p}   \,d x
+\int_{B_{4R} }|f|(w-\ell)_-\,dx\nonumber\\
&\qquad
+\int_{B_{4R} } \int_{\mathbb{R}^N \backslash  B_{2R}}  \frac{ (w(x) -\ell)_-^{q-1}(w(y)-\ell)_-}{\left|x-x_0\right|^{N+s q}} \,d x d y
\bigg)\nonumber\\
&\leq C
\bigg(\frac{\ell^{p}}{(4 R)^p}+\frac{\ell^{q}}{(4 R)^{sq}}
+\ell\| f\|_{L^\gamma\left(B_{4 R}\right)}\left|B_{4 R}\right|^{-\frac{1}{\gamma}}\bigg)\left|B_{4 R}\right|\nonumber\\
&
\qquad+C\left(\frac{\ell^q}{(4R)^{sq}}+\frac{\ell\operatorname{Tail}^{q-1}\left(w_-;x_0, 4 R\right)}{R^{s q}}\right)|B_{4R}|
\nonumber\\
&\leq C H_{4R}(\ell)\left|B_{4 R}\right|.
\end{align}
Here by the nonnegativity of $w$ in $B_{4 R}$,
 \begin{align*}
\left\|(w-\ell)_{-}\right\|^m_{L^m\left(B_{4R}\right)}\le C\ell^m|B_{4R}|  \quad \text{for any } m\ge1.
 \end{align*}

Next, according to the terms on the left-hand side of \eqref{h1},
we let $\tau=1/2$ and distinguish two mutually exclusive cases:
\begin{align}\label{cases}
\left(\frac{\delta^\tau  t}{R}\right)^p>\left(\frac{\delta^\tau  t}{R^s}\right)^q  \quad \text {and} \quad 
\left(\frac{\delta^\tau  t}{R}\right)^p \leq \left(\frac{\delta^\tau  t}{R^s}\right)^q.
\end{align}

 {\bf Case $\eqref{cases}_2$:} Let us put $\ell=4\delta^{\tau} t$ in \eqref{h1} and find that
\begin{align*}
\int_{B_{2R}} & \int_{B_{2R}} \frac{(w(x)-4 \delta^\tau t)_{+}^{q-1}(w(y)-4 \delta^\tau t)_{-}}{|x-y|^{N+sq}} \,d x d y
\le C\left(\frac{\delta^\tau  t}{R^s}\right)^q|B_{4R}|
\end{align*}
and
$$
\begin{aligned}
&\quad \int_{B_{2R}} \int_{B_{2R}} \frac{(w(x)-4 \delta^\tau t)_{+}^{q-1}(w(y)-4 \delta^\tau t)_{-}}{|x-y|^{N+s q}} \,d x d y \\
& \geq \frac{1}{(4R)^{N+sq}  }
\int_{B_{2R} \cap\{u \geq t\}}(w(x)-4 \delta^\tau t)^{q-1} \,d x \int_{B_{2R} \cap\{w<2 \delta^\tau t\}}(4 \delta^\tau t-w(y)) \,d y \\
& \geq \frac{ \delta^\tau t^q}{  CR^{sq} }
\frac{\left|B_{2R} \cap\{w \geq t\}\right|}{ |B_{2R}| }     \left|B_{2R} \cap\{w<2 \delta^\tau t\}\right| \\
& \geq \frac{ \delta^\tau t^q \nu}{C R^{sq} }\left|B_{2R} \cap\{w<2 \delta^\tau t\}\right|
\end{aligned}
$$
with $C>1$ depending on $p,q,s,N,\Lambda,a_0,A_0$, where we used the assumption \eqref{hc0}
and the fact that  $4\delta^\tau \leq 1 / 2$ and $|x-y|^{N+s q} \leq (4R)^{N+sq}$, for
any $x, y \in B_{2R}$. By virtue of two above estimates, we readily get
$$
\frac{\left|B_{2R} \cap\{w<2 \delta t\}\right|}{ |B_{2R}|  } \leq
\frac{\left|B_{2R} \cap\{w<2 \delta^\tau t\}\right|}{ |B_{2R}|  } \leq \frac{C }{\nu}\delta^{\tau(q-1)},
$$
as expected.

 {\bf Case $\eqref{cases}_1$:} Let $m \geq 7$ be the unique integer for which
$$
2^{-m-1} \leq \delta<2^{-m} .
$$
Consider the decreasing sequence $\{2^{-k} t\}_{k=0}^m$.
Notice that $2^{-k} t \in(2 \delta t, t]$ for any $k \in\{0, \ldots, m-1\}$.
Moreover, by \eqref{hc0}, it is easy to see that
 for $k\in\{1, \ldots, m-2\}$,
\begin{align}\label{ho9}
\left|B_{2R} \cap\left\{\left(w-2^{-k+1}t \right)_{-} \leq 2^{-k}t\right\}\right|
=\left|B_{2R} \cap\left\{w\geq 2^{-k}t\right\}\right|
\geq\left|B_{2R} \cap\{w\geq t\}\right| \geq \nu\left|B_{2R}\right|
\end{align}
and
\begin{align}\label{ho10}
 \left|B_{2R}\cap\left\{\left(w-2^{-k+1}t\right)_{-} \geq 3 \cdot 2^{-k-1} t\right\}\right|
=\left|B_{2R} \cap\left\{w \leq 2^{-k-1}t\right\}\right| .
\end{align}
The case \eqref{cases}$_1$ yields that
$$R^{sq-p}
> (\delta^{\tau} t)^{q-p}
> (2^{(-m-1)\tau} t)^{q-p}
>(2^{-k-1} t)^{q-p},\quad k=k_0, \ldots, m-2$$
with $k_0$ being the smallest integer bigger than $\tau m$.
This ensures that
$$ \frac{(2^{-k+1} t)^{q}}{R^{s q}}\le\frac{ C(2^{-k+1} t)^{p}  }{R^p},\quad k=k_0, \ldots, m-2.$$
Since $2^{-k+1} \ge \delta $, 
it follows from the last display and \eqref{h1} with $\ell=2^{-k+1} t$ that
\begin{align}\label{ho11}
  \left[(w-2^{-k+1} t)_{-}\right]_{W^{1,p}\left(B_{2R}\right)}^p
& \leq  CH_{4 R}(2^{-k+1} t)\left|B_{4 R}\right| \nonumber\\
& \leq C\left(\frac{(2^{-k+1} t)^{p}}{R^p}+ \frac{(2^{-k+1} t)^q}{R^{s q}}\right)\left|B_{4 R}\right|, \nonumber\\
&\leq C  R^{N-p} (2^{-k+1} t)^{p},\quad k=k_0, \ldots,m-2.
\end{align}
 Consequently, we can apply Lemma \ref{measure1} to the function $\left(w-2^{-k+1} t\right)_{-}$, with $h_1=2^{-k} t$
 and $h_2=3 \cdot 2^{-k-1} t$. We easily get
\begin{align*}
&\quad\Bigg(  \frac{\big|B_{2R} \cap\{\left(w-2^{-k+1} t\right)_{-} \leq 2^{-k} t\} \big|}{ |B_R|  }
\cdot
\frac{\big| B_{2R} \cap\{\left(w-2^{-k+1} t\right)_{-} \geq 3 \cdot 2^{-k-1} t\}\big|}{ |B_R| }
\Bigg)^{\frac{N-1}{N}}\\
&\leq \frac{C 2^{k}R^{1-\frac{N}{p}} }{t}\left[\left(w-2^{-k+1} t\right)_{-}\right]_{W^{1, p}\left(B_{2R}\right)}
\Bigg(
\frac{\big|B_{2R} \cap\{2^{-k} t<\left(w-2^{-k+1} t\right)_{-}<3 \cdot 2^{-k-1} t\}\big|}{    |B_R|    }
\Bigg)^{\frac{p-1}{p}},
\end{align*}
which combined with \eqref{ho9} and \eqref{ho10} yields that
\begin{align*}
& \quad \Bigg(\frac{\left|B_{2R} \cap\left\{w \leq 2^{-k-1} t\right\}\right|}{ |B_{2R}| }
\Bigg)^{\frac{N-1}{N}}\\
& \leq \frac{CR^{1-\frac{N}{p}} 2^k}{\nu^{\frac{N-1}{N}} t}
 \left[\left(w-2^{-k+1} t\right)_{-}\right]_{W^{1, p}\left(B_{2R}\right)}
 \Bigg(\frac{\big|B_{2R}  \cap\left\{2^{-k-1} t<w<2^{-k} t\right\}\big|}{  |B_{2R}|  }
 \Bigg)^{\frac{p-1}{p}}
\end{align*}
for some $C >0$ depending only on $N$ and $p$. We can control the Gagliardo seminorm of $\left(w-2^{-k+1}t\right)_{-}$ according to
\eqref{ho11} and deduce that, for any $k \in\{k_0, \ldots, m-2\}$,
$$
\Bigg(\frac{\left|B_{2R} \cap\left\{w \leq 2^{-k-1} t\right\}\right|}{ |B_{2R}| }
\Bigg)^{\frac{N-1}{N}}
 \leq
 \frac{C }{\nu^{\frac{N-1}{N}} }
 \Bigg(\frac{\big|B_{2R}  \cap\left\{2^{-k-1} t<w<2^{-k} t\right\}\big|}{  |B_{2R}|  }
 \Bigg)^{\frac{p-1}{p}}.
$$
By adding up the above inequality as $k$ ranges between $k_0$ and $m-2$,
we find
$$
(m-2-k_0)\Big(\frac{\left|B_{2R} \cap\{w<2 \delta t\}\right|}{ |B_{2R}| }  \Big)^{\frac{(N-1) p}{N(p-1)}}
 \leq  \frac{C}{\nu^{\frac{(N-1) p}{N(p-1)}}    }
 \sum_{i=k_0}^{m-2}
 \frac{\left|B_{2R} \cap\left\{2^{-k-1} t<w<2^{-k} t\right\}\right|}{  |B_{2R}| }
 \leq  \frac{C}{ \nu^{\frac{(N-1) p}{N(p-1)}}    },
$$
which in turn yields that
$$
\frac{\left|B_{2R} \cap\{w<2 \delta t\}\right|}{ |B_{2R}| }
 \leq  \frac{C}{ \nu  }|\log (\delta^{1-\tau})|^{-\frac{N(p-1)}{(N-1) p}}.
$$
The proof is therefore complete.
\end{proof}

Based on the information of measure theory above, we can get the following pointwise result:

\begin{lemma}\label{lem4.2}
Assume that $K_{sq}$ and $F$ satisfy \eqref{F} and \eqref{K} with $a_0>0$. Let $B_{4 R}:=B_{4R}(x_0) \subset\subset\Omega$ with $R \leq 1$.
Let $u \in \mathcal{A}(\Omega) \cap L_{s q}^{q-1}\left(\mathbb{R}^N\right)$ be a minimizer of \eqref{q1}, where the function $f|_{B_{4 R}}$ belongs to $L^{\gamma}(B_{4 R})$ with $\gamma>\max\{\frac{N}{p},1\}$. Assume that for some $M\in\Bbb{R}$, $u\ge M$ in $B_{4 R}$ and
$$
\left|B_{2 R} \cap\{u-M \geq t\}\right| \geq \nu\left|B_{2 R}\right|
$$
with some $\nu \in(0,1)$ and $t>0$. Then there exists $\delta\in\left(0,\frac{1}{2^8}\right]$, which depends only on the absolute constants $N, p,q,s,\Lambda,a_0,A_0$ and $\nu$, such that whenever
\begin{align}\label{hol1}
\| f\|_{L^\gamma\left(B_{4 R}\right)}\left|B_{4 R}\right|^{-\frac{1}{\gamma}}
+(4R)^{-sq}\operatorname{Tail}^{q-1}\left((u-M)_{-} ; x_0, 4 R\right)
\leq h_{4R}(\delta t)
\end{align}
with $h_{4 R}(\delta t):= \frac{(\delta t)^{p-1}}{(4 R)^p}+ \frac{(\delta t)^{q-1}}{(4 R)^{sq}}$ as determined in \eqref{ha},
then we can find that
\begin{align}\label{hol3l}
u-M\geq \delta t \quad \text { in } B_R.
\end{align}
\end{lemma}

\begin{proof}[\bf Proof.]
We still set $w:=u-M$.
Let $\delta \in(0, 2^{-8}]$ and $\epsilon \in\left(0,2^{-N-1}\right]$ to be specified later.
We initially suppose that
\begin{align}\label{hol6}
\left|B_{2R} \cap\{w<2 \delta t\}\right| \leq \epsilon\left|B_{2R}\right|
\end{align}
with sufficiently small $\epsilon$.
We arbitrarily choose radii $\rho,r$ satisfying $R/2\le r/2<\rho<r\le 2R$.
In view of \eqref{hol6} with $\epsilon\le 2^{-N-1}$, we have that
for any $k \in[\delta t, 2\delta t]$ and
$$
\begin{aligned}
\left|B_\rho \cap\left\{(w-k)_{-}=0\right\}\right|
& =\left|B_\rho \backslash\{w<k\}\right|
\geq \left|B_\rho\right|-\left|B_{2R} \cap\{w<2 \delta t\}\right| \\
& \geq \left(1-\epsilon\left(\frac{2R}{\rho}\right)^N\right)
\left|B_\rho\right|
 \geq \left(1-2^N \epsilon\right)\left|B_\rho\right| \\
& \geq \frac{1}{2}\left|B_\rho\right| .
\end{aligned}
$$
This enables us to utilize the Poincar\'{e}-Sobolev inequality and find that
\begin{align}\label{hol9}
\left(k-h\right)^q\left(\frac{\left|B_{\rho} \cap\left\{w<h\right\}\right|}{\left|B_{\rho }\right|}\right)^{\frac{q}{q_s^*}}
& \leq \left(\dashint_{B_{\rho} } (w-k)_{-}^{q_s^*} \, d x\right)^{\frac{q}{q_s^*}} \nonumber\\
& \leq C
 \rho^{ s q} \dashint_{B_\rho } \int_{B_\rho }\left|(w(x)-k)_{-} -(w(y)-k)_{-} \right|^q \,d\mu
\end{align}
and
\begin{align}\label{hol9'}
\left(k-h\right)^p
\left(\frac{\left|B_{\rho} \cap\left\{w<h\right\}\right|}{\left|B_{\rho }\right|}\right)^{\frac{ p}{p^*}}
& \leq\left(\dashint_{B_{\rho} } (w-k)_{-}^{p^*}\, dx\right)^{\frac{ p}{p^*}} \nonumber\\
& \leq C
 \rho^{ p} \dashint_{B_{\rho}} |D\big( (w-k)_{-} \big)|^p \,d x
\end{align}
for any $h\in(\delta t,k)$.
With setting
$$\kappa:=\frac{p^*}{p}<\frac{q_s^*}{q},$$
we derive from \eqref{hol9}, \eqref{hol9'} and Lemma \ref{lemCa} that
\begin{align}\label{hol11'}
&\quad\bigg( \left(\frac{k-h}{\rho^s}\right)^q+\left(\frac{k-h}{\rho}\right)^p \bigg)
\left(\frac{\left|B_{\rho} \cap\left\{w<h\right\}\right|}{\left|B_{\rho }\right|}\right)^{\frac{1}{\kappa}} \nonumber\\
& \leq C \dashint_{B_\rho } \int_{B_\rho }\left|(w(x)-k)_{-} -(w(y)-k)_{-} \right|^q\, d \mu
+\dashint_{B_{\rho}} |D\big( (w-k)_{-} \big)|^p \,d x
\nonumber\\
&\leq C\left(\frac{ r}{r-\rho}\right)^{N+q}
\bigg( \dashint_{B_r} \frac{(w-k)_{-}^q }{r^{sq}} \,dx
+\dashint_{B_{r}}  \frac{(w-k)_{-} ^p}{ r^p}   \,d x\nonumber\\
&\qquad\qquad\qquad\qquad+
\dashint_{B_r} \int_{\mathbb{R}^N \backslash  B_{\rho}}  \frac{a(x,y)(w(x)-k)_{-}^{q-1}(w(y)-k)_{-}}{\left|x-x_0\right|^{N+s q}} \,d x d y
+\dashint_{B_{r}}|f|(w-k)_{-}\bigg),
\end{align}
where the left-hand right can be estimated as below,
\begin{align}
\bigg( \left(\frac{k-h}{\rho^s}\right)^q+\left(\frac{k-h}{\rho}\right)^p \bigg)
\left(\frac{\left|B_{\rho} \cap\left\{w<h\right\}\right|}{\left|B_{\rho }\right|}\right)^{\frac{1}{\kappa}}
\ge H_{4R}(k-h)\left(\frac{\left|B_{\rho} \cap\left\{w<h\right\}\right|}{\left|B_{\rho }\right|}\right)^{\frac{1}{\kappa}} .
\end{align}
Then we utilize the following fact
\begin{align}
 \left\|(w-k)_{-}\right\|_{L^m\left(B_{r}\left(x_0\right)\right)}^m
\le C k^{m}\left|A^{-}\left(k,x_0,r\right)\right|\text{  for any  } m\ge1
\end{align}
to estimate first two integrals in \eqref{hol11'}.
Due to \eqref{hol1}, we can see
\begin{align}\label{hol11}
&\quad\int_{B_{r} } \int_{\mathbb{R}^N \backslash  B_{\rho}}  \frac{a(x,y)(w(x)-k)_{-}^{q-1}(w(y)-k)_{-}}{\left|x-x_0\right|^{N+s q}}  \,d x d y\nonumber\\
&\leq \int_{B_{r} } (w(y)-k)_-\bigg(\int_{\mathbb{R}^N \backslash  B_{4R}}  \frac{a(x,y)w_-^{q-1}(x)}{\left|x-x_0\right|^{N+s q}} \,d x\bigg) d y\nonumber\\
&\quad+
\int_{B_{r} } (w(y)-k)_-\bigg(\int_{\mathbb{R}^N \backslash  B_{2R}} \frac{k^{q-1}}{\left|x-x_0\right|^{N+s q}} \,d x\bigg) d y\nonumber\\
& \leq C k\left|A(k,x_0,r)\right| \operatorname{Tail}^{q-1}\left(w_-; x_0, 2 R\right)
+Ck^q\left|A(k,x_0,r)\right|\left(R^{-p}+ R^{-s q}\right)\nonumber\\
&\leq CH_{4R}(k)\left|A(k,x_0,r)\right|.
\end{align}
Here the positive constant $C$ depends on $N,p,q,s,A_0$ and $\Lambda$.
Moreover, we estimate $\int_{B_{r}}|f|(w-k)_{-}\,dx$ as below,
\begin{align}\label{hol11''}
\int_{B_{r}}|f|(w-k)_{-} \,dx&
\leq k\| f\|_{L^\gamma\left(B_{r}\right)}\left|A^{-}(k,x_0, r)\right|^{\frac{\gamma-1}{\gamma}}\nonumber\\
&
=k\| f\|_{L^\gamma\left(B_{r}\right)}\left|B_{r}\right|^{-\frac{1}{\gamma}}
\frac{\left|A^{-}(k,x_0, r)\right|^{\frac{\gamma-1}{\gamma}}}{   \left|B_{r}\right|^{-\frac{1}{\gamma}}  } \nonumber\\
& \leq k h_{4 R}(\delta t) \frac{\left|A^{-}(k,x_0, r)\right|^{\frac{\gamma-1}{\gamma}}}{   \left|B_{r}\right|^{-\frac{1}{\gamma}}  }\nonumber\\
&
\leq H_{4 R}(k) \frac{\left|A^{-}(k,x_0, r)\right|^{\frac{\gamma-1}{\gamma}}}{   \left|B_{r}\right|^{-\frac{1}{\gamma}}  }
\end{align}
because of \eqref{hol1}.
Combining \eqref{hol11'}--\eqref{hol11''} tells that
$$
\left(\frac{\left|A^{-}(h, x_0,\rho)\right|}{\left|B_\rho\right|}\right)
\leq C\left(\frac{r}{r-\rho}\right)^{(N+q) \kappa}
\left(\frac{H_{4R}(k)}{H_{4 R}(k-h)}\right)^\kappa\left(\frac{\left|A^{-}(k, x_0,r)\right|}{\left|B_r\right|}\right)^{\frac{\gamma-1}{\gamma} \kappa}.
$$
Consider the sequences $\left\{r_i\right\}_{i=0}^\infty$ and $\left\{k_i\right\}_{i=0}^\infty$
defined by
$$r_i:=(1+2^{-i})R \quad \text{and} \quad  k_i:=\left(1+2^{-i}\right) \delta t.$$
Also set $Y_i:=\left|A^{-}\left(k_i, x_0, r_i\right)\right| /\left|B_{r_i}\right|$.
By applying \eqref{hol9} with $h=k_i, k=k_{i-1}, \rho=r_i$ and $r=r_{i-1}$, we obtain that
\begin{align*}
Y_{i+1} \leq C 2^{i(N+2 q) \kappa} Y_i^{ \frac{\kappa(\gamma-1)}{\gamma}},\qquad i=1,2,3, \ldots.
\end{align*}
The assumption on $\gamma$ ensures that
$$ \frac{\kappa(\gamma-1)}{\gamma}>1.$$
For $p \geq N$, we could take $\kappa$ larger than $\frac{\gamma}{\gamma-1}$ ahead of time. Then in order to exploit the convergence lemma, we force
$$
Y_0=\frac{|A^{-}\left(2\delta t, x_0, 2R\right)|}{\left|B_{2 R}\right|}
\leq (2C)^{-\frac{1}{\kappa(\gamma-1) /\gamma-1}}2^{-\frac{\kappa(N+2q)}{(\kappa(\gamma-1) /\gamma-1)^2}}
=:\theta,
$$
which can be realized by choosing $\delta$ sufficiently small. Through Lemma \ref{lem4.1},
 we select $\delta \in\left(0, \frac{1}{2^8}\right]$, that depends on
$ p, q, s, N,\Lambda$ and $a_0,A_0$, such that
$$
\frac{C }{\nu}\Big(\delta^{(q-1)/2}
+ |\log \delta|^{-\frac{N(p-1)}{(N-1) p}}\Big)
 \leq \min\{\theta,2^{-N-1}\} .
$$
Then we can infer from this display and Lemma \ref{lem-iteration} that
$$
\lim _{i \rightarrow \infty} Y_i=0,
$$
which directly guarantees that $w \geq \delta t$ in $B_R$.
Hence, under the above choice of $\delta$ in the statement of this lemma, the positivity expansion result \eqref{hol3l} follows clearly.
\end{proof}

We now end this section by giving the proof of H\"{o}lder continuity, in which we need pay much attention to the constant $\delta$ in Lemma \ref{lem4.2} independent of the arbitrary number $M$. For this reason, it is possible to get the desired result as follows.

\medskip

\noindent{\bf Proof of Theorem \ref{th3}.}
Let $\delta \in(0,2^{-8}]$ be the constant found in Lemma \ref{lem4.2}.
By Lebesgue's dominated convergence theorem, we can find small $\alpha$ satisfying
\begin{align}\label{hold1}
0<\alpha \leq \min \left\{\frac{s}{2 }\, ,\, \log_4\left(\frac{2}{2-\delta}\right)\, ,\, \frac{p\gamma-N }{2\gamma(p-1)}\right\}
\end{align}
and
\begin{align}\label{hold2}
\int_4^{+\infty} \frac{\left(\rho^\alpha-1\right)^{q-1}}{\rho^{1+ sq}} d \rho
 \leq \frac{\delta^{q-1}}{32^{q+1} N \left|B_1\right|} .
\end{align}
Then we set
\begin{align}\label{hold3}
j_0:=\max\Bigg\{\frac{2}{sq} \log _4\left(\frac{32^{q+1} N(1+\left|B_1\right|)}{s \delta^{q-1}}\right),
 \frac{2\gamma(p-1)}{p\gamma - N} \log _4\left(\frac{4}{ \delta|B_1|^{\frac{1}{\gamma(p-1)}}}\right)
 \Bigg\} .
\end{align}
What follows is to utilize induction arguments to prove that there exist a non-decreasing sequence $\left\{m_i\right\}_{i=0}^\infty$
and a non-increasing sequence $\left\{M_i\right\}_{i=0}^\infty$ of real numbers such that
\begin{align}\label{hold4}
 m_i \leq u \leq M_i\quad \text{   in }B_{4^{1-i} R},\qquad i=0,1,2,3, \ldots
 \end{align}
 and
 \begin{align}\label{hold5}
 M_i-m_i=4^{-\alpha i} L,\qquad i=0,1,2,3, \ldots
 \end{align}
  with
\begin{align}\label{hold6}
 L:=2 \cdot 4^{\frac{s j_0}{2 }}
 \|u\|_{L^{\infty}(B_{4 R(x_0)})}+\operatorname{Tail}(u ; x_0, 4 R)
 +\|f\|^{\frac{1}{p-1}}_{L^\gamma(B_{4 R(x_0)})}.
  \end{align}
Let us take $m_i:=-4^{-\alpha i} L / 2$ and $M_i:=4^{-\alpha i} L / 2$, for any $i=0, \ldots, j_0$.
Then, \eqref{hold4} holds for these $i$ 's, thanks to \eqref{hold3} and \eqref{hold6}.
Now we fix an integer $j \geq j_0$ and suppose that the sequences $\left\{m_i\right\}_{i=1}^{j}$
 and $\left\{M_i\right\}_{i=1}^{j}$ have been constructed.
Our expected claim \eqref{hold4} will be proved once we find proper $m_{j+1}$ and $M_{j+1}$.

Let us define the function
$$
v:=\frac{2 \cdot 4^{\alpha j}}{L}\left(u-\frac{M_j+m_j}{2}\right)\quad \text{   in }\mathbb{R}^N .
$$
By \eqref{hold4}, \eqref{hold5} and the monotonicity of $\left\{m_i\right\}_{i=1}^{j}$, $\left\{M_i\right\}_{i=1}^{j}$,
we can obtain that
\begin{align}\label{hold6'}
\left|M_j+m_j\right| \leq \left(1-4^{-\alpha j}\right) L .
\end{align}
Since $\max\{2u-M_j-m_j, M_j+m_j-2u\}\le M_j-m_j $ in $B_{4^{1-j} R}$,
it is clear that
$$|v|\le\frac{2 \cdot 4^{\alpha j}}{L}\left(\frac{M_j-m_j}{2}\right)\quad \text{   in } B_{4^{1-j} R},$$
then
$$|v| \leq 1\quad \text{   in } B_{4^{1-j} R}.$$
Take $x \in B_{4 R} \backslash B_{4^{1-j} R}$ and
let $\ell \in\{0, \ldots, j-1\}$ be the unique integer for which
$x \in B_{4^{1-\ell} R} \backslash B_{4^{-\ell} R}$. By virtue of \eqref{hold4}, \eqref{hold5}
and the monotonicity of $\left\{m_i\right\}_{i=1}^{j}$,
 we have
$$
\begin{aligned}
v(x) & \leq \frac{2 \cdot 4^{\alpha j}}{L}\left(M_{\ell}-m_{\ell}+m_{\ell}-\frac{M_j+m_j}{2}\right)\\
&\leq \frac{2 \cdot 4^{\alpha j}}{L}\left(M_{\ell}-m_{\ell}+m_j-\frac{M_j+m_j}{2}\right) \\
& =\frac{2 \cdot 4^{\alpha j}}{L}\left(M_{\ell}-m_{\ell}-\frac{M_j-m_j}{2}\right)=2 \cdot 4^{\alpha(j-\ell)}-1 \\
& \leq 2\left( \frac{4^j|x|}{R}\right)^\alpha-1 .
\end{aligned}
$$
An application of similar arguments ensures that $v(x) \geq-2\left(4^j|x| / R\right)^\alpha+1$, and hence
\begin{align}\label{hold10}
(1 \pm v(x))_{-}^{q-1} \leq 2^{q-1}\left(\left( \frac{4^j|x|}{R}\right)^\alpha-1\right)^{q-1}
\quad \text { for a.a. } x \in B_{4 R} \backslash B_{4^{1-j} R} .
\end{align}
Meanwhile, we can derive from \eqref{hold6'} that
\begin{align}\label{hold11}
(1 \pm v(x))_{-}^{q-1} \leq 2^{q-1}\left(\left(\frac{2 \cdot 4^{\alpha j}}{L}\right)^{q-1}|u|^{q-1}+4^{\alpha(q-1) j}\right)
\quad \text { for a.a. }\mathbb{R}^N \backslash B_{4 R}.
\end{align}
With the help of \eqref{hold10}, \eqref{hold11} and changing variables appropriately, we obtain
\begin{align*}
& \quad\operatorname{Tail}\left((1 \pm v)_{-} ; x_0, 4^{1-j} R\right)^{q-1} \\
&\leq 4^{-j s q+sq+q-1} R^{s q}
\Bigg(\int_{\mathbb{R}^N \backslash B_{4^{1-j}R} } \frac{ \Big(\left( 4^j|x|/R\right)^\alpha-1\Big)^{q-1}}{|x|^{N+s q}} \,d x\\
& \quad+\left(\frac{4^{\alpha j}}{L}\right)^{q-1} \int_{\mathbb{R}^N \backslash B_{4 R}} \frac{|u(x)|^{q-1}}{|x|^{N+s q}} \,d x
+4^{\alpha(q-1) j} \int_{\mathbb{R}^N \backslash B_{4 R}} \frac{\,d x}{|x|^{N+s q}}\Bigg)
\\
& \leq
 8^q N \left|B_1\right|\int_4^{+\infty} \frac{\left(\rho^\alpha-1\right)^{q-1}}{\rho^{1+sq}} d \rho
 +8^q 4^{\left(\alpha q-sq\right) j}  \frac{\operatorname{Tail}(u ; x_0,4 R)^{q-1}}{L^{q-1}}
\\
& \quad+ \frac{8^q N \left|B_1\right|4^{\left(\alpha q-sq\right) j}}{sq }\\
& \leq
 8^q N \left|B_1\right|\int_4^{+\infty} \frac{\left(\rho^\alpha-1\right)^{q-1}}{\rho^{1+sq}} d \rho
 + \frac{8^{q+1} N  (\left|B_1\right|+1)4^{\left(\alpha q-sq\right) j}}{s}.
\end{align*}
As a consequence of \eqref{hold2}, \eqref{hold3} and \eqref{hold6}, it holds that
\begin{align}\label{hold11'}
\operatorname{Tail}\left((1 \pm v)_{-} ; x_0, 4^{1-j} R\right) \leq \frac{\delta}{4} .
\end{align}
Now, we have that either
 \begin{align}\label{hold12}
 \left|B_{4^{1-j} R / 2} \cap\{v \geq 0\}\right| \geq \frac{1}{2}\left|B_{4^{1-j} R / 2}\right|
 \text { or } \left|B_{4^{1-j} R / 2} \cap\{v \geq 0\}\right|<\frac{1}{2}\left|B_{4^{1-j} R / 2}\right|.
 \end{align}
If the first choice of \eqref{hold12} happens, we consider the function
$$\frac{L}{2\cdot 4^{\alpha j}}(1+v)
=u-\frac{M_j+m_j}{2}+\frac{L}{2\cdot 4^{\alpha j}}.$$
Then, we have
$$
\left|B_{4^{1-j} R / 2} \cap
\Big\{u-\frac{M_j+m_j}{2}+\frac{L}{2\cdot 4^{\alpha j}} \geq \frac{L}{2\cdot 4^{\alpha j}}\Big\}\right|
=\left|B_{4^{1-j} R / 2} \cap\{v \geq 0\}\right|
 \geq\frac{1}{2}\left|B_{4^{1-j} R / 2}\right| .
$$
It follows by \eqref{hold3} and \eqref{hold6} that
\begin{align}\label{hold11'}
\| f\|_{L^\gamma\left(B_{4 ^{1-j}R}\right)}\left|B_{4 ^{1-j}R}\right|^{-\frac{1}{\gamma}}
\le|B_1|^{-\frac{1}{\gamma}}\| f\|_{L^\gamma\left(B_{4 R}\right)}(4 ^{1-j}R)^{-\frac{N}{\gamma}}\le
\Big(\frac{1}{4 ^{-j}R}\Big)^{p}\Big(\frac{\delta}{2}\cdot \frac{L}{2\cdot 4^{\alpha j}}\Big)^{p-1}.
\end{align}
Moreover, there holds that
$$\text {Tail}\left( \Big(u-\frac{M_j+m_j}{2}+\frac{L}{2\cdot 4^{\alpha j}}\Big)_{-} ; x_0, 4^{1-j} R\right)
=\frac{L}{2\cdot 4^{\alpha j}}\text {Tail}\left((1 + v)_{-} ; x_0, 4^{1-j} R\right)
\le\frac{\delta}{4}\cdot \frac{L}{2\cdot 4^{\alpha j}}.$$
We utilize Lemma \ref{lem4.2} with $t=\frac{L}{2\cdot 4^{\alpha j}}$ and
$M=\frac{M_j+m_j}{2}-\frac{L}{2\cdot 4^{\alpha j}} $ to find that
$$u-\frac{M_j+m_j}{2}+\frac{L}{2\cdot 4^{\alpha j}}
\ge\frac{L}{2\cdot 4^{\alpha j}}\delta\quad \text{   in }  B_{ 4^{-j}R }.$$
This directly tells that
\begin{align*}
u&\ge\frac{M_j+m_j}{2}+\frac{L}{2\cdot 4^{\alpha j}}\delta-\frac{L}{2\cdot 4^{\alpha j}}\\
&=M_j-\frac{M_j-m_j}{2}-\frac{L}{2\cdot 4^{\alpha j}}(1-\delta)\\
&=M_j-\frac{L}{2\cdot 4^{\alpha j}}(2-\delta)\quad \text{   in }  B_{ 4^{-j}R }.
\end{align*}
In view of \eqref{hold1}, we can derive that
$$
M_j-4^{-(j+1) \alpha} L \leq u \leq M_j \quad \text{   in }  B_{4^{-j} R}.
$$
This directly guarantees \eqref{hold4} for $i=j+1$ with $M_{j+1}:=M_j$ and $m_{j+1}:=M_{j+1}-4^{-(j+1) \alpha} L$.
If instead the second alternative in \eqref{hold12} holds, we shall deal with the function $\frac{(1-v)L}{2\cdot 4^{\alpha j} }$ and utilize an analogous argument as above to obtain the same conclusion by taking $m_{j+1}:=m_j$
and $M_{j+1}:=m_{j+1}+4^{-(j+1) \alpha} L$.
\qquad$\Box$

\section{Harnack inequality}

This section is devoted to establishing Harnack estimates on the minimizers of \eqref{q1}. The forthcoming lemma can be inferred in a very similar way to Lemma \ref{lem4.2} with $M=0$.

\begin{lemma}\label{lem5.1}
Assume $K_{sq}$ and $F$ satisfy \eqref{F} and \eqref{K} with $a_0>0$. Let $u \in \mathcal{A}(\Omega)\cap L_{sq}^{q-1}\left(\mathbb{R}^N\right)$
be a minimizer of \eqref{q1} which is nonnegative in a ball $B_{16 R}:=B_{16R}(x_0) \subset\subset\Omega$ with $R \leq 1$.
The function $f|_{B_{16 R}}$ belongs to $L^{\gamma}(B_{16 R})$ with $\gamma>1$. Suppose that
$$
\left|B_{ R} \cap\{u \geq t\}\right| \geq \nu^{k}\left|B_{R}\right|
$$
for some $\nu \in(0,1)$, $t>0$ and $k\in \mathbb{N}^+$. Then there exists $\delta \in\left(0, \frac{1}{2^8}\right]$ that depends only on $N,p,q,s,a_0,A_0$ and $\nu$, if
$$
\| f\|_{L^\gamma\left(B_{16 R}\right)}\left|B_{16 R}\right|^{-\frac{1}{\gamma}}
+(16R)^{-sq}\operatorname{Tail}^{q-1}\left(u_{-} ; x_0, 16 R\right)
\leq h_{16R}(\delta^kt)
$$
with $h_{16 R}(\delta t):= \frac{(\delta^k t)^{p-1}}{(16 R)^p}+\frac{(\delta^k t)^{q-1}}{(16 R)^{sq}}$ as given in \eqref{ha},
then there holds that
\begin{align*}
u\geq \delta^kt \quad \text { in } B_R.
\end{align*}
\end{lemma}

With Lemma \ref{lem5.1} at hands, we can conclude the following weak Harnack inequality.

\begin{lemma}\label{lem5.2}
 Assume that $K_{sq}$ and $F$ satisfy \eqref{F} and \eqref{K} with $a_0>0$.
Let $u \in \mathcal{A}(\Omega) \cap L_{s q}^{q-1}\left(\mathbb{R}^N\right)$, nonnegative in a ball $B_{16 R}:=B_{16R}(x_0) \subset\subset\Omega$ with $R \leq 1$, be a minimizer of \eqref{q1}.
Suppose that the function $f|_{B_{16 R}}$ belongs to $L^{\gamma}(B_{16 R})$ with $\gamma>\max\Big\{1,\frac{N}{p}\Big\}$.
Then there exist constants $\varepsilon_0 \in(0,1)$ and $C \geq 1$, both depending on $s,p,q,N,\Lambda,a_0$ and $A_0$,
such that
\begin{align}\label{wh}
\left(\dashint_{B_R} u^{\varepsilon_0} \,d x\right)^{\frac{1}{\varepsilon_0}}
\leq C \inf _{B_R} u+C h_{ 16R }^{-1}
\left(d
+(16R)^{-sq}\operatorname{Tail}^{q-1}\left(u_{-} ; x_0, 16 R\right)\right),
\end{align}
where $ d:=\Big(\dashint_{B_{16R}}f^\gamma(x)\,dx\Big)^{\frac{1}{\gamma}}$.
\end{lemma}
\begin{proof}
[\bf Proof.]
Let $\delta \in\left(0, \frac{1}{2^8}\right]$ be the constant determined in Lemma \ref{lem4.2} under the choice $\nu=\frac{1}{2}$.
We accordingly set
\begin{align}\label{e}
\varepsilon_0:=\frac{\log \nu}{2 \log \delta}=\frac{1}{2 \log _{\frac{1}{2}} \delta} \in(0,1).
\end{align}
We claim that for any $t \geq 0$,
\begin{align}\label{inf}
\inf _{B_R} u
+
h_{16R}^{-1}
\left(d+(16R)^{-sq}\operatorname{Tail}^{q-1}\left(u_{-} ; x_0, 16 R\right)\right)
 \geq \delta\left(\frac{\left|A^{+}(t, x_0,R)\right|}{\left|B_R\right|}\right)^{\frac{1}{2 \varepsilon_0}} t .
\end{align}
We only consider the case $t \in\left[0, \sup _{B_R} u\right)$. Otherwise the above inequality holds trivially.

For each $t \in\left[0, \sup _{B_R} u\right)$, let $k=k(t)$ be the unique integer fulfilling
\begin{align}\label{kt}
\log _{\frac{1}{2}} \frac{\left|A^{+}(t,x_0, R)\right|}{\left|B_R\right|}
\leq k<1+\log _{\frac{1}{2}} \frac{\left|A^{+}(t, x_0,R)\right|}{\left|B_R\right|} .
\end{align}
Notice that \eqref{e} and \eqref{kt} indicate
$$
\delta^k \geq \delta\left(\frac{\left|A^{+}(t, x_0,R)\right|}{\left|B_R\right|}\right)^{\frac{1}{2 \varepsilon_0}}.
$$
Now we only consider the case that
\begin{align*}
d
+(16R)^{-sq}\operatorname{Tail}^{q-1}\left(u_{-} ; x_0, 16 R\right)
<h_{16R }\left(\delta^k t\right).
\end{align*}
Otherwise it is easy to see \eqref{inf} is true based on the definition of $k$.
Still by \eqref{kt}, there holds that
$$
\left|A^{+}(t,x_0, R)\right| \geq 2^{-k}\left|B_R\right|,
$$
which in conjunction with Lemma \ref{lem5.1} infers that
$$
u \geq \delta^k t \text { in } B_R,
$$
and so
$$
\inf _{B_R} u
+
h_{16R}^{-1}
\left(d
+(16R)^{-sq}\operatorname{Tail}^{q-1}\left(u_{-} ; x_0, 16 R\right)\right)
 \geq \delta^k t .
$$
This along with \eqref{kt} guarantees \eqref{inf}.
At this moment, a similar argument as in the proof of \cite[Proposition 6.8]{Co} deduces the desired result.
\end{proof}

Next, we combine the supremum estimate in Theorem \ref{th1}, the weak Harnack inequality in Lemma \ref{lem5.2} and the tail estimate below to infer Harnack estimate, Theorem \ref{th4}.

\medskip

\noindent{\bf Proof of Theorem \ref{th4}.}
We first are going to obtain a tail estimate.
Fix any $z\in B_R(x_0)\subset\subset\Omega$ and $r\in(0,2R]$.
By denoting $M:=\sup _{B_r(z)} u>0$, we apply Lemma \ref{lemCa} with $k \equiv 2 M$ to get that
\begin{align}\label{har1}
 &\int_{B_{r / 2}(z)}(u(x)-2 M)_{-}\left(\int_{\mathbb{R}^N} \frac{(u(y)-2 M)_{+}^{q-1}}{|x-y|^{N+sq}} d y\right) \,d x \nonumber\\
\leq & C\left(\frac{\left\|(u-2 M)_{-}\right\|_{L^q\left(B_r(z)\right)}^q}{r^{s q}}
+ \frac{\left\|(u-2 M)_{-}\right\|_{L^p\left(B_r(z)\right)}^p}{r^p}\right)\nonumber\\
& +\frac{C}{r^{sq}}\left\|(u-2 M)_{-}\right\|_{L^1\left(B_r(z)\right)}
 \operatorname{Tail}^{q-1}\left((u-2 M)_{-} ; z, r / 2\right) \nonumber\\
&
 +C\|(u-2 M)_{-}\|_{L^{\gamma'}\left(B_r(z)\right)}\|f\|_{L^{\gamma}\left(B_r(z)\right)}.
\end{align}
It is easy to find out that
$$
(u(y)-2 M)_{+}^{q-1} \geq \min \left\{1,2^{2-q}\right\} u^{q-1}_{+}(y)-2^{q-1} M^{q-1} .
$$
From the above two observations and the fact that $u \leq M$ on $B_r(z)$, it follows that
\begin{align}\label{har1'}
& \quad\int_{B_{r/2}(z)}(u(x)-2 M)_{-}\left(\int_{\mathbb{R}^N} \frac{(u(x)-2 M)_{+}^{q-1}}{|x-y|^{N+sq }} d y\right)\,d x \nonumber\\
& \geq 2^{-N-sq} M \int_{B_{r/2
}(z)}
\left(\int_{\mathbb{R}^N\backslash B_r(z)} \frac{\min \left\{1,2^{2-q}\right\} u^{q-1}_{+}(y)-2^{q-1} M^{q-1}}{|y-z|^{N+sq}} d y\right) \,d x \nonumber\\
& \geq \frac{M r^{N-s q}}{C} \operatorname{Tail}\left(u_{+} ; z, r\right)-C r^{N-s q} M^q,
\end{align}
where we utilized the fact
$|x-y| \leq 2|y-z|$ for any $x \in B_r(z)$ and $y \in \mathbb{R}^N \backslash B_r(z)$.
On the other hand, since $u \geq 0$ on $B_r(z)$, we have
\begin{align}\label{har2}
\frac{\left\|(u-2 M)_{-}\right\|_{L^q\left(B_r(z)\right)}^q}{r^{s q}}
+ \frac{\left\|(u-2 M)_{-}\right\|_{L^p\left(B_r(z)\right)}^p}{r^p}
\le
r^{N-s p}(r^{s q-p}M^p+ M^q).
\end{align}
The last two terms can be estimated as:
\begin{align}\label{har3}
&\frac{1}{r^{sq}}\left\|(u-2M)_{-}\right\|_{L^1\left(B_r(z)\right)} \operatorname{Tail}^{q-1}\left((u-2M)_{-};z,r/2\right) \nonumber\\
&\quad+\left\|(u-2 M)_{-}\right\|_{L^{\gamma'}\left(B_r(z)\right)}\|f\|_{L^{\gamma}\left(B_r(z)\right)} \nonumber\\
&\le Mr^{N-sq} \operatorname{Tail}^{q-1}\left(u_-;z,r\right)+Mr^{\frac{N}{\gamma'}}\|f\|_{L^{\gamma}\left(B_r(z)\right)} \nonumber\\
&\le Mr^{N-sq} \operatorname{Tail}^{q-1}(u_-;z,r)+Mr^Nd.
\end{align}
Substituting \eqref{har1'}, \eqref{har2} and \eqref{har3} into \eqref{har1} infers that
\begin{align*}
Mr^{N-sq} \operatorname{Tail}^{q-1}\left(u_+ ; z, r\right)
 \leq Cr^{N-sq}\left(M\operatorname{Tail}^{q-1}\left(u_{-} ; z, r\right)+Mr^{sq}d
 +r^{s q-p}M^p+ M^q\right),
\end{align*}
which directly ensures the following tail estimate that
\begin{align*}
\operatorname{Tail}^{q-1}\left(u_{+} ; z, r\right)
 & \leq C\big(\operatorname{Tail}^{q-1}\left(u_{-} ; z, r\right)+r^{s q}h_r(M)+r^{s q}d\big).
\end{align*}

By applying the reverse operator $h_r^{-1}$, we obtain
\begin{align*}
h_{r}^{-1}\big(r^{-sq}\operatorname{Tail}^{q-1}\left(u_{+} ; z, r\right)\big)
 & \leq Ch_{r}^{-1}\big(r^{-sq}\operatorname{Tail}^{q-1}\left(u_{-} ; z, r\right)\big)+M+h_{r}^{-1}(d).
\end{align*}
From Lemma \ref{lem3.3}, one can say that for any $\delta_1$,
\begin{align*}
\sup _{B_{r}(z)} u_{+}
&\leq C_{\delta_1} H_{2r}^{-1}\left(\dashint_{B_{ 2r}} H_{2r}\left(u_{+}\right) \,d x\right)
+\delta_1 h_{2r}^{-1}\big((2r)^{-sq}\operatorname{Tail}^{q-1} (u_+;x_0,r)\big)
+\delta_1 g_{2r}^{-1}(d)
\end{align*}
with some $C_{\delta_1}>0$ depending on $\delta_1$.
A combination of the above two estimates implies that
\begin{align}\label{har4}
\sup _{B_r(z)} u & \leq C_{\delta_1} H_{2r}^{-1}\left(\dashint_{B_{ 2r}} H_{2r}\left(u_{+}\right) \,d x\right)
+C\delta_1 h_{r}^{-1}\big(r^{-sq}\operatorname{Tail}^{q-1}\left(u_{-} ; z, r\right)\big)
+\delta_1 g_{2r}^{-1}(d)\nonumber\\
& \quad+ C{\delta_1} M
+h_{r}^{-1}(d)\nonumber\\
&\leq C_{\delta_1} H_{r}^{-1}\left(\dashint_{B_{ 2r}} H_{r}\left(u_{+}\right) \,d x\right)
+C\delta_1 \operatorname{Tail}(u_-;z,r)
+\delta_1 g_{r}^{-1}(d)+ C{\delta_1} M
+h_{r}^{-1}(d),
\end{align}
where we used the regularity of the functions $H_{r},g_{r},h_{r}$.
Using Jensen's inequality with the convex function $t \mapsto\left[H_{r}^{-1}(t)\right]^q$,
we can see that for any $\delta_2>0$,
\begin{align}\label{har5}
H_{r}^{-1}\left(\dashint_{B_{ 2r}(z)} H_{r}\left(u_{+}\right) \,d x\right)
& \leq\left(\dashint_{B_{2 r}(z)} u^q \,d x\right)^{\frac{1}{q}} \nonumber\\
& \leq\left(\sup _{B_{2 r}(z)} u\right)^{\frac{q-\varepsilon_0}{q}}
\left(\dashint_{B_{2 r}(z)} u^{\varepsilon_0} \,d x\right)^{\frac{1}{q}} \nonumber\\
& \leq \delta_2 \sup _{B_{2 r}(z)} u+C_{\delta_2}\left(\dashint_{B_{2 r}(z)} u^{\varepsilon_0} \,d x\right)^{\frac{1}{\varepsilon_0}}
\end{align}
with $\varepsilon_0>0$ determined in Lemma \ref{lem5.2}.
With taking $\delta_1, \delta_2$ sufficiently small, we derive from \eqref{har4} and \eqref{har5} that
\begin{align}\label{har6}
\sup _{B_r(z)} u& \leq \frac{1}{2} \sup _{B_{2 r}(z)} u
+C\left(\dashint_{B_{2 r}(z)} u^{\varepsilon_0} \,d x\right)^{\frac{1}{\varepsilon_0}}
+C \operatorname{Tail}\left(u_{-} ; z, r\right)+Ch_{r}^{-1}(d_r)+Cg_{r}^{-1}(d_r)\nonumber\\
&\leq \frac{1}{2} \sup _{B_{2 r}(z)} u
+C\left(\dashint_{B_{2 r}(z)} u^{\varepsilon_0} \,d x\right)^{\frac{1}{\varepsilon_0}}
+C \operatorname{Tail}\left(u_{-} ; z, r\right)\nonumber\\
&\quad+C\big(r^{-N+p}\|f\|_{L^\gamma(B_r(z))}\big)^{\frac{1}{p-1}}+C\big(r^{-N+\bar{q}}\|f\|_{L^\gamma(B_r(z))}\big)^{\frac{1}{\bar{q}-1}}.
\end{align}
By employing \eqref{har6} along with a suitable covering argument and the technical lemma \cite[Lemma 4.11]{Co}, we arrive at
\begin{align*}
\sup _{B_R} u &\leq \left(\dashint_{B_{2R}} u^{\varepsilon_0} \,d x\right)^{\frac{1}{\varepsilon_0}}
+ \operatorname{Tail}\left(u_{-} ; x_0,R\right)\\
&\quad+R^{(p-\frac{N}{\gamma})/(p-1)}\|f\|^{\frac{1}{p-1}}_{L^\gamma(B_{2R})}
+R^{(\bar{q}-\frac{N}{\gamma})/(\bar{q}-1)}\|f\|^{\frac{1}{\bar{q}-1}}_{L^\gamma(B_{2R})}.
\end{align*}
This together with \eqref{wh} yields the desired Harnack's inequality.\qquad$\Box$

 \end{document}